\documentclass{amsart}

\usepackage{amsmath,amscd}
\usepackage{amssymb}
\usepackage{amsfonts}
\usepackage{amsthm}
\usepackage{verbatim}

 \numberwithin{equation}{section}
\newcommand{\nc}{\newcommand}
\nc{\nt}{\newtheorem}
\nc{\dmo}{\DeclareMathOperator}
\nc{\enm}{\ensuremath}

\newtheorem{thm}{Theorem}
\newtheorem{conj}{Conjecture}
\newtheorem{prop}{Proposition}
\newtheorem{lemma}{Lemma}
\nt{defn}{Definition}
\newtheorem{cor}{Corollary}
\nt{assumption}{Assumption}

\dmo{\Ind}{Ind} \dmo{\cInd}{c-Ind} \dmo{\Adj}{Ad} \dmo{\PGL}{PGL} \dmo{\SO}{SO} \dmo{\Lie}{Lie} \dmo{\Int}{Int} \dmo{\reg}{reg}
\dmo{\sing}{sing} \dmo{\supp}{supp} \dmo{\tr}{tr} \dmo{\Sym}{Sym} \dmo{\Hom}{Hom} \dmo{\Tor}{Tor} \dmo{\Out}{Out} \dmo{\Ht}{ht} \dmo{\End}{End}
\dmo{\Mat}{Mat} \dmo{\Tr}{Tr} \dmo{\tmp}{temp}

\dmo{\SL}{SL} \dmo{\sgn}{sgn} \dmo{\GL}{GL}

\dmo{\GSp}{GSp} \dmo{\gsp}{gsp}
\dmo{\disc}{disc}
 \dmo{\simp}{sc}
\dmo{\alg}{alg}
 \dmo{\Hol}{Hol}
 \dmo{\largess}{Large}
 
 \dmo{\Sp}{Sp}
\dmo{\kc}{kc}  \dmo{\ac}{ac}
\dmo{\ck}{ck}
\dmo{\syp}{sp}
\dmo{\REG}{reg}
 \DeclareMathOperator{\Res}{Res} \dmo{\Mod}{mod} \dmo{\geo}{geo} \dmo{\re}{Re} \dmo{\Spec}{Spec}
\dmo{\Fr}{Fr} \dmo{\vol}{vol} \dmo{\Sets}{Sets} \dmo{\im}{im} \dmo{\diag}{diag} \dmo{\Ker}{Ker} \dmo{\val}{val} \dmo{\ord}{ord}
\dmo{\Stab}{Stab} \dmo{\Ad}{Ad} \dmo{\rank}{rank} \dmo{\Symp}{Sp} \dmo{\Nm}{Nm} \dmo{\Norm}{Norm}
\dmo{\Top}{top}
\nc{\bG}{\enm{{\mathbf G}}}

\nc{\Aff}{\mathbb{A}}

\nc{\eps}{\varepsilon}
 \nc{\isom}{\stackrel{\sim}{\to}}
\nc{\ups}{g_\theta}

\nc{\bks}{\enm{{\backslash}}}

 \nc{\Z}{\enm{{\mathbb Z}}}

\nc{\Zp}{\enm{{\mathbb Z_p}}}

\nc{\Gm}{\enm{{\mathbb G_m}}}
\nc{\F}{\enm{{\mathbb F}}}
\nc{\Fp}{\enm{{\mathbb F}_p}}
\nc{\Fq}{\enm{{\mathbb F}_q}}
\nc{\Q}{\enm{{\mathbb Q}}}
\nc{\Qp}{\enm{{\mathbb Q_p}}}
\nc{\R}{\enm{{\mathbb R}}}
\nc{\N}{\enm{{\mathbb N}}}
\nc{\C}{\enm{{\mathbb C}}}
\nc{\CC}{\enm{{\mathcal C}}}
\nc{\half}{\enm{{\frac{1}{2}}}}
\nc{\BB}{\enm{{\mathcal B}}}
\nc{\flip}{\tilde{\eps}}
\nc{\funE}{\mathcal E}
\nc{\ii}{\enm{{\mathcal I}}}
\nc{\jj}{\enm{{\mathcal J}}}
\nc{\OO}{\enm{{\mathcal O}}}
\nc{\f}{\enm{{\mathcal F}}}

\nc{\GGl}{\enm{{\mathfrak gl}}}
\nc{\GG}{\enm{{\mathfrak g}}}
\nc{\gd}{\enm{{\hat{\mathfrak g}}}}
\nc{\gm}{\enm{{\gamma}}}
\nc{\hh}{\enm{{\mathfrak h}}}

\nc{\II}{\enm{{\mathfrak a}}}
\nc{\LL}{\enm{{\mathfrak l}}}
\nc{\mm}{\enm{{\mathfrak m}}}
\nc{\pp}{\enm{{\mathfrak p}}}
\nc{\TT}{\enm{{\mathfrak t}}}
\nc{\Nc}{\enm{{\mathcal N}}}
\nc{\Cc}{\enm{{\mathcal C}}}
\nc{\MM}{\enm{{\mathcal M}}}

\nc{\iso}{\tilde{\rightarrow}}
\dmo{\pr}{pr}
\dmo{\der}{der} \dmo{\ad}{ad}
\dmo{\cok}{cok}
\dmo{\Gal}{Gal}
\dmo{\cont}{cont}
\dmo{\Fib}{Fib}

\nc{\Acok}{A_{\cok}}

\nc{\Aker}{A_{\ker}}
\nc{\Tcok}{T_{\cok}}

\nc{\Tker}{T_{\ker}}

\nc{\res}{res}

 \nc{\gmt}{\tilde{\gm}}
\nc{\vtil}{\tilde{v}}

 \nc{\Gd}{\enm{{\hat{G}}}}

  \nc{\Hd}{\enm{{\hat{H}}}}

\nc{\vt}{\enm{\vartheta}}
\nc{\lra}{\enm{\longrightarrow}}
\nc{\ra}{\enm{\rightarrow}}
\nc{\lip}{\enm{\langle}}
\nc{\rip}{\enm{\rangle}}
\nc{\Kot}{\mathcal{K}}
\nc{\bsk}{\bigskip}

 \nc{\ol}{\overline}

\nc{\tchi}{\chi'}

\author{Steven Spallone}

\begin{document}

\title   {Stable Trace Formulas and Discrete Series Multiplicities }
\maketitle
 
\begin{abstract}
Let $G$ be a reductive algebraic group over $\Q$, and suppose that $\Gamma \subset G(\R)$ is an arithmetic subgroup defined by congruence conditions.   A basic problem in arithmetic is to determine the multiplicities of discrete series representations in $L^2(\Gamma \bks G(\R))$, and in general to determine the traces of Hecke operators on these spaces.  In this paper we give a conjectural formula for the traces of Hecke operators, in terms of stable distributions.  It is based on a stable version of Arthur's formula for $L^2$-Lefschetz numbers \cite{A-L^2}, which is due to Kottwitz \cite{SVAF}. We reduce this formula to the computation of elliptic p-adic orbital integrals and the theory of endoscopic transfer.   As evidence for this conjecture, we demonstrate the agreement of the central terms of this formula with the unipotent contributions to the multiplicity coming from Selberg's trace formula in Wakatsuki \cite{Wak Mult}, in the case $G=\GSp_4$ and $\Gamma=\GSp_4(\Z)$.   
 
 \end{abstract}

\section{Introduction}
 
 Let $G$ be a reductive algebraic group over $\Q$, and $\Gamma$  an arithmetic subgroup of $G(\R)$ defined by congruence conditions.  Then $G(\R)$ acts on $L^2(\Gamma \bks G(\R))$ via right translation; let us write $R$ for this representation.   A fundamental problem in arithmetic is to understand $R$.   
As a first step, we may decompose $R$ as
\begin{equation*}
R= R_{\disc} \oplus R_{\cont},
\end{equation*}
where $R_{\disc}$ is a direct sum of irreducible representations, and $R_{\cont}$ decomposes continuously.  The continuous part may be understood inductively through Levi subgroups of $G$ as in \cite{Langlands c}.  We are left with the study of $R_{\disc}$.  Given an irreducible representation $\pi$ of $G(\R)$, write $R_{\disc}(\pi)$ for the $\pi$-isotypic subspace of $R_{\disc}$.  Then
\begin{equation*}
R_{\disc}(\pi) \cong \pi^{\oplus m_{\disc}(\pi)}
\end{equation*}
for some integer $m_{\disc}(\pi)$.  (We may also write $m_{\disc}(\pi,\Gamma)$.)  A basic problem is to compute these integers.  

There is more structure than simply these dimensions, however.  Arithmetic provides us with a multitude of Hecke operators $h$ on  $L^2(\Gamma \bks G(\R))$ which commute with $R$.  Write $R_{\disc}(\pi,h)$ for the restriction of $h$ to  $R_{\disc}(\pi)$.  The general problem is to find a formula for the trace of $R_{\disc}(\pi,h)$.

We focus on discrete series representations $\pi$.  These are representations which behave like representations of compact or finite groups, in the sense that their associated matrix coefficients are square integrable.  Like other smooth representations, they have a theory of characters developed by Harish-Chandra.  They separate naturally into finite sets called $L$-packets.  For an irreducible finite-dimensional representation $E$ of $G(\C)$, there is a corresponding $L$-packet $\Pi_E$ of discrete series representations, consisting of those with the same infinitesimal and central characters as $E$.  Say that a discrete series representation is regular if $\pi \in \Pi_E$, with the highest weight of $E$ regular.

We follow the tradition of computing $\tr R_{\disc}(\pi,h)$ through trace formulas.  This method has gone through several incarnations, beginning in a paper \cite{Selberg} of Selberg's for $\GL_2$, in which he also investigated the continuous Eisenstein series.  A goal was to compute dimensions of spaces of modular forms, and traces of Hecke operators on these spaces.  These spaces of modular forms correspond to the spaces $R_{\disc}(\pi)$ we are discussing in this case.  His trace formula is an integral, over the quotient of the upper half space $X$ by $\Gamma$, of a sum of functions $H_{\gm}$, one for each element of $\Gamma$.  Let us write it roughly as

\begin{equation*}
\dim_{\C} S(\Gamma)= \int_{\Gamma \bks X} \sum_{\gm \in \Gamma} H_{\gm}(Z) dZ,
\end{equation*}
for some space $S(\Gamma)$ of cusp forms with a suitable $\Gamma$-invariance condition.

Here $dZ$ is a $G(\R)$-invariant measure on $X$.  When the quotient $\Gamma \bks X$ is compact, the sum and integral may be interchanged, leading to a simple expression for the dimensions in terms of orbital integrals.  The interference of the Eisenstein series precludes this approach in the noncompact quotient case.  Here there are several convergence difficulties, which Selberg overcomes by employing a truncation process.  Unfortunately the truncation process leads to notoriously complicated expressions which are far from being in closed form.  This study of $R_{\disc}(\pi)$ has been expanded to other reductive groups using what is called the Arthur-Selberg trace formula.
(See Arthur \cite{Arthur Clay}.)

Generally, a trace formula is an equality of distributions on $G(\R)$, or on the adelic group $G(\Aff)$.  One distribution is called the geometric side; it is a sum of terms corresponding to conjugacy classes of $G$.  Given a test function $f$, the formula is essentially made up of combinations $I_M(\gm,f)$ of weighted integrals of $f$ over the conjugacy classes of elements $\gm$.  (Here $M$ is a Levi subgroup of $G$.)  Another distribution is called the spectral side, involving the Harish-Chandra transforms $\tr \pi(f)$ for various representations $\pi$.   Here, the operator $\pi(f)$ is given by weighting the representation $\pi$ by $f$.  The geometric and spectral sides agree, and in applications we can learn much about the latter from the former.  Some of the art is in picking test functions to extract information about both sides.

The best general result using the trace formula to study $\tr R_{\disc}(\pi,h)$ seems to be in Arthur \cite{A-L^2}.  He produces a formula for
\begin{equation} \label{stableR}
\sum_{\pi \in \Pi}  \tr R_{\disc}(\pi,h),
\end{equation}
where $\Pi$ is a given discrete series $L$-packet for $G(\R)$.   He uses test functions $f$ which he  calls ``stable cuspidal''.   Their Fourier transforms $\pi \mapsto \tr \pi(f)$ are ``stable'' in the sense that they are constant on $L$-packets, and ``cuspidal" in the sense that, considered as a function defined on tempered representations, they are supported on discrete series.  (Tempered representations are those which appear in the Plancherel formula for $G(\R)$.)  Using his invariant trace formula (Arthur \cite{ITF1}, \cite{ITF2}), he obtains (\ref{stableR}) as the spectral side.  The geometric side is a combination of  orbital integrals for $h$ and values of Arthur's $\Phi$-function, which describes the asymptotic values of discrete series characters averaged over an $L$-packet.
 
In particular, he produces a formula for
\begin{equation} \label{stablemult}
\sum_{\pi \in \Pi} m_{\disc}(\pi),
\end{equation}
for a $L$-packet $\Pi$ of (suitably regular) discrete series representations.
 
In the case of $G=\GL_2$, there is a discrete series representation $\pi_k$ for each integer $k \geq 1$.  In this case $m_{\disc}(\pi_k)$ is the dimension of the space $S_k(\Gamma)$ of $\Gamma$-cusp forms of weight $k$ on the upper half plane.  Restriction to $\SL_2(\R)$ gives two discrete series $\{ \pi_k^+, \pi_k^- \}$ in each $L$-packet.  However we may still use Arthur's formula here since $m_{\disc}(\pi_k^+, \Gamma)=m_{\disc}(\pi_k^-, \Gamma)$ for every arithmetic subgroup $\Gamma$.  (Endoscopy does not play a role.)
    
For the group $\GSp_4(\R)$ there are two discrete series representations in each $L$-packet: one ``holomorphic" and one ``large" discrete series.  Let $\pi$ be a holomorphic discrete series, and write $\pi'$ for the large discrete series representation in the same $L$-packet as $\pi$.  The multiplicity $m_{\disc}(\pi,\Gamma)$ is also the dimension of a certain space of vector-valued Siegel cusp forms  (see \cite{Wallach}) on the Siegel upper half space, an analogue of the usual cusp forms on the upper half plane.  For $\Gamma=\Sp_4(\Z)$, the dimensions of these spaces of cusp forms were calculated in Tsushima \cite{Tsu 1}, \cite{Tsu 2}  using the Riemann-Roch-Hirzebruch formula, and later in Wakatsuki \cite{Wak Dim} using the Selberg trace formula and the theory of prehomogeneous vector spaces.   Wakatsuki then evaluated Arthur's formula in \cite{Wak Mult} to compute 
\begin{equation*}
m_{\disc}(\pi,\Gamma)+m_{\disc}(\pi',\Gamma),
\end{equation*}
thereby deducing a formula for $m_{\disc}(\pi',\Gamma)$.
 
A natural approach to isolating the individual $m_{\disc}(\pi)$, or generally the individual $\tr R_{\disc}(\pi,h)$, is to apply a trace formula to a matrix coefficient, or more properly, a pseudocoefficient $f$.   This means that $f$ is a test function whose Fourier transform picks out $\pi$ rather than the entire packet $\Pi $ containing $\pi$.   (See Definition \ref{pseudoscience} below.)  Such a function will not be stable cuspidal, but merely cuspidal.  Arthur (\cite{A-L^2}, see also \cite{Arthur Clay}) showed that $I_M(\gm,f)$ vanishes when $f$ is stable cuspidal and the unipotent part of $\gm$ is nontrivial.  If we examine the geometric side of Arthur's formula for a pseudocoefficient $f$, 
we must evaluate the more complicated terms $I_M(\gm,f)$ for elements $\gm$ with nontrivial unipotent part.  
At the time of this writing, such calculations have not been made in general; we take another approach.
 
Distinguishing the individual representations $\pi$ from others in its $L$-packet leads to the theory of endoscopy, and stable trace formulas.  The grouping of representations $\pi$ into packets $\Pi$ on the spectral side mirrors the fusion of conjugacy classes which occurs when one extends the group $G(\R)$ to the larger group $G(\C)$.  If $F$ is a local or global field, then a stable conjugacy class in $G(F)$ is, roughly, the union of classes which become conjugate in $G(\ol{F})$.  (See \cite{Langlands defn} for a precise definition.) 

The distribution which takes a test function to its integral over a regular semisimple stable conjugacy class is a basic example of a stable distribution.  Indeed, a stable distribution is defined to be a closure of the span of such distributions (see \cite{Langlands debuts}, \cite{Langlands defn}).  A distribution on $G(F)$ is stabilized if it can be written as a sum of stable distributions, the sum being over smaller subgroups $H$ related to $G$.  These groups $H$ are called endoscopic groups for $G$; they are tethered to $G$ not as subgroups but through their Langlands dual groups.  As part of a series of techniques called endoscopy, one writes unstable distributions on $G$ as combinations of stable distributions on the groups $H$.  Part of this process is the theory of transfer, associating suitable test functions $f^H$ on $H(F)$ to test functions $f$ on $G(F)$ which yield a matching of orbital integrals.  Indeed this was the drive for \cite{Ngo}.
As the name suggests, the theory of endoscopy, while laborious, leads to an intimate understanding of $G$.

There has been much work in stabilizing Arthur's formula: See for example \cite{Langlands debuts}, \cite{A-STF1}, \cite{A-STF2}, \cite{A-STF3}. In Kottwitz's preprint \cite{SVAF}, he defines a stable version of Arthur's Lefschetz formula, which we review below.  (See also \cite{Morel}.) It is a combination
\begin{equation*}
\Kot(f)=\sum_H \iota(G,H) ST_g(f^H)
\end{equation*}
of distributions $f \mapsto f^H \mapsto ST_g(f^H)$ over endoscopic groups $H$ for $G$.  Here the distributions $ST_g$, defined for each $H$, are stable.  (See Section  \ref{Various Invariants} for the definition of the rational numbers $\iota(G,H)$.)  Each $ST_g$ is a sum of terms corresponding to stable conjugacy classes of elliptic elements $\gm \in H(\Q)$.  The main result of \cite{SVAF} is that $\Kot$ agrees with Arthur's distribution, at least for functions $f$ which are stable cuspidal at the real place.   

As part of the author's thesis \cite{Sp}, he evaluated the identity terms of Kottwitz's distribution for the group $G=\SO_5$ at a function $f$ which was a pseudocoefficient for a discrete series representation at the real place.  Later, Wakatsuki noted that the resulting expressions match up with the terms in his multiplicity formulas $m_{\disc}(\pi, \Gamma)$ and $m_{\disc}(\pi',\Gamma)$ corresponding to unipotent elements.  Moreover, the contribution in \cite{Sp} from the endoscopic group accounted for the difference in these multiplicity formulas, while the stable part corresponded to the sum.  After further investigation, we conjecture simply that Kottwitz's distribution evaluated at a function $f=f_{\pi,\Gamma}$ suitably adapted to $\pi$ and $\Gamma$ is equal to $m_{\disc}(\pi,\Gamma)$, under a regularity condition on $\pi$.  (See Section \ref{Conjecture} for the precise statement.)   Of course this is compatible with Arthur's results in \cite{A-L^2}.  
 
In this paper we give some computational evidence for this conjecture.  We also reduce the computation of each $ST(f_{\pi,\Gamma}^H)$ to evaluating elliptic orbital $p$-adic integrals for the transfer $f^{\infty H}$ at the finite places.  The rest breaks naturally into a problem at the real points and a global volume computation. 

The main ingredient at the archimedean place is Arthur's $\Phi$-function $\Phi_M(\gm, \Theta^E)$, which we review.  This quantity gives the contribution from the real place to the trace formulas in \cite{A-L^2} and \cite{GKM}.  It also plays a prominent role in Kottwitz's formula. This function, originally defined by the asymptotic behaviour of a stable character near a singular elliptic element $\gm$, was expressed in closed form by the author in \cite{SpalPhi}.    

There are two volume-related constants that enter into any explicit computation of $ST_g$.  The first is $\ol{v}(G)$, which is essentially the volume of an inner form of $G$ over $\R$.  It depends on the choice of local measure $dg_{\infty}$. The second comes about from orbital integrals at the finite ideles, and depends on the choice of local measure $dg_f$.   These integrals may frequently be written in terms of the volumes of open compact subgroups $K_f$ of $G(\Aff_f)$.  In practice, one is left computing expressions such as $\ol{v}(G)^{-1} \vol_{dg_f}(K_f)^{-1}$, which are independent of the choice of local measures.  More specifically, we define
\begin{equation*}
\chi_{K_f}(G)=\ol{v}(G)^{-1} \vol_{dg_f}(K_f)^{-1} \tau(G) d(G).
\end{equation*}
Here $\tau(G)$ is the Tamagawa number of $G$ and $d(G)$ is the index of the real Weyl group in the complex Weyl group.
 A main general result of this paper, Theorem \ref{chi-theorem}, interprets $\chi_{K_f}(G)$ via Euler characteristics of arithmetic subgroups.  It extends a computation of Harder \cite{Harder}, which was for semisimple simply connected groups, to the case of reductive groups, under some mild hypotheses on $G$.

We work out two examples in this paper, one for $\SL_2$ and another for $\GSp_4$.   It is easy to verify our conjecture for $G=\SL_2$ and $\Gamma=\SL_2(\Z)$ using the classic dimension formula for cusp forms.  In this case endoscopy does not appear.   The calculations for $\GSp_4$ are more complex; we content ourselves with working out the central terms of Kottwitz's formula.  

If $\pi$ is a holomorphic discrete series representation of $\GSp_4(\R)$, write $H_{1}^{\pi}$ for the central-unipotent terms of the Selberg trace formula, as evaluated in \cite{Wak Mult} to compute $m_{\disc}(\pi,\Gamma)$.  Here $\Gamma=\GSp_4(\Z)$.  If $\pi$ is a large discrete series representation, write $H_{1}^{\pi}$ for the central-unipotent terms in \cite{Wak Mult} contributing to $m_{\disc}(\pi,\Gamma)$.  In both cases, write $f=f_{\pi,\Gamma}=f_{\infty} f^{\infty}$, with $f_{\infty}$ a pseudocoefficient for $\pi$, and $f^{\infty}$ the (normalized) characteristic function of the integer adelic points of $G$.  Write $\Kot(f, \pm 1)$ for the central terms of Kottwitz's formula applied to $f$.

As evidence for our conjecture,  we show:  

\begin{thm} \label{Theorem}
For each regular discrete series representation $\pi$ of $G(\R)$ we have 
\begin{equation*}
\Kot(f_{\pi,\Gamma}, \pm 1)=H_1^{\pi}.
\end{equation*}
\end{thm}  

We believe that the central terms of Kottwitz's formula will generally match up with the difficult central-unipotent terms of the Arthur-Selberg formula, as in this case.

Our conjecture reduces the computation of discrete series multiplicities to the computation of stable elliptic orbital integrals of various transfers $f_p^H$, written for functions on $G(\Qp)$.  Let us write this as $SO_{\gm_H}(f_p^H)$.  Here $f_p$ are characteristic functions of congruence subgroups of $G(\Qp)$ related to $\Gamma$. 
Certainly at suitably regular elements, $SO_{\gm_H}(f_p^H)$ is an unstable combination of orbital integrals of $f_p$, however there are also contributions from elliptic singular $\gm_H$, notably $\gm_H=1$.  At present, there are expressions for $f_p^H$ in the parahoric case and of course for $G(\Zp)$, but less seems to be known for smaller congruence subgroups.  On the other hand, there are many formulas for dimensions of Siegel cusp forms and discrete series multiplicities for these cases  (e.g. \cite{Wak Mult}).  This suggests that one could predict stable singular elliptic orbital integrals $SO_{\gm_H}(f_p^H)$ for the transfer $f_p^H$ of characteristic functions of congruence subgroups (e.g. Klingen, Iwahori, Siegel), by comparing our formulas.

We now describe the layout of this paper. 

In Section \ref{prelim} we set up the conventions for this study.  We explain how we are setting up the orbital integrals, and indicate our main computational tools.  We also review the Langlands correspondence for real groups.

The theory of Arthur's $\Phi$-function is reviewed in Section \ref{APF}.  In Section \ref{Kottwitz's Formula}, we review Kottwitz's stable version of Arthur's formula from \cite{SVAF}.  We also state our conjecture here.  The heart of the volume computations in this paper is in Section \ref{Euler Characteristics}, where we determine $\chi_K(G)$.
As a warm-up, we work out the classic case of $\SL_2$, with $\Gamma=\SL_2(\Z)$ in Section \ref{Special Linear}.   

The case of $G=\GSp_4$ is considerably more difficult.  We must work out several isomorphisms of real tori.   These are described in Section \ref{Real Tori}.
The basic structure of $G$ and its Langlands dual $\hat{G}$ is set up in Section \ref{Structure of}.
In Section \ref{Discrete Series for G} we work out the Langlands parameters for discrete series of $G(\R)$.
There is only one elliptic endoscopic group $H$ for $G$.  We describe $H$ in Section \ref{Endoscopic}.
 In Section \ref{Discrete Series for H}, we describe the Langlands parameters for discrete series of $H(\R)$ and describe the transfer of discrete series in this case.
In Section \ref{Levi Subgroups} we describe the Levi subgroups of $G$ and $H$ and compute various constants which occur in Kottwitz's formula for these groups.
In Section \ref{Computing for G Levis} we compute explicitly Arthur's $\Phi$-function for Levi subgroups of $G$, and we do this for Levi subgroups of $H$ in Section \ref{Computing for H Levis}.  In Section \ref{Final Central}, we write out the terms of Kottwitz's formula corresponding to central elements of $G$ and $H$, for a general arithmetic subgroup $\Gamma$.  In Section \ref{Integral Points}, we specialize to the case of $\Gamma=\GSp_4(\Z)$, and in Section \ref{Comparison} we gather our results to demonstrate Theorem \ref{Theorem}.

\bigskip

{\bf Acknowledgements:}  This paper is founded on the author's thesis under the direction of Robert Kottwitz.  The author would like to thank him for his continual help with this project.   This paper is also indebted to Satoshi Wakatsuki for predicting Theorem \ref{Theorem}, and for much useful correspondence.  We would also like to thank 
 Ralf Schmidt for helpful conversations.

\section{Preliminaries and Notation} \label{prelim}
We denote by $\Aff$ the ring of adeles over $\Q$.  We denote by $\Aff_f$ the ring of finite adeles over $\Q$, so that $\Aff=\Aff_f \times \R$.  Write $\OO_f$ for the integral points of $\Aff_f$.

If $G$ is a real Lie group, we write $G^+$ for the connected component of $G$ (using the classical topology rather than any Zariski topology). 

Let $G$ be a connected reductive group over $\R$.  A torus $T$ in $G$ is elliptic if $T/A_G$ is anisotropic (as an $\R$-torus).  Say that $G$ is cuspidal if it contains a maximal torus $T$ which is elliptic.  An element of $G(\R)$ is elliptic if it is contained in an elliptic maximal torus of $G$.  Having fixed an elliptic maximal torus $T$, the absolute Weyl group $\Omega_G$ of $T$ in $G$ is the quotient of the normalizer of  $T(\C)$ in $G(\C)$ by $T(\C)$. The real Weyl group $\Omega_{G,\R}$ of $T$ in $G$ is the quotient of the normalizer of $T(\R)$ in $G(\R)$ by $T(\R)$.  We may drop the subscript ``$G$'' if it is clear from context.  Also fix a maximal compact subgroup $K_{\R}$ of $G(\R)$.

Write $q(G)$ for half the dimension of $G(\R)/K_{\R}Z(\R)$.
If we write $R$ for the roots of $G$, with a set of positive roots $R^+$, then
\begin{equation*}
q(G)=\half  \left( |R^+|+ \dim(X) \right), 
\end{equation*}
where $X$ is the span of $R$.

If $F$ is a field, write $\Gamma_F$ for the absolute Galois group of $F$.  Suppose $G$ is an algebraic group over $F$.  If $E$ is an extension field of $F$, we write $G_E$ for $G$ viewed as an algebraic group over $E$ (by restriction).  If  $\gm$ is an element of $G(F)$, we denote by $G_\gm$ the centralizer of $\gm$ in $G$.  By $G^\circ$ we denote the
identity component of $G$ (with the Zariski topology).  Write $G_{\der}$ for the derived group of $G$.  If $G$ is a reductive group, write $G_{\simp}$ for the simply connected cover of $G_{\der}$.  Let $X^*(G)=\Hom(G_{\ol F},\Gm)$ and $X_*(G)=\Hom(\Gm,G_{\ol F})$.   These are abelian groups.  Write $X^*(G)_{\C}$ and $X_*(G)_{\C}$ for the tensor product of these groups over $\Z$ with $\C$.   Similarly with the subscript ``$\R$".  Write $A_G(F)$ for the maximal $F$-split torus in the center of $G$. 

If $G$ is an algebraic group over $\Q$, let $G(\Q)^+=G(\R)^+ \cap G(\Q)$.

\subsection{Endoscopy}

In this section we review the theory of based root data and endoscopy in the form we will use in this paper. 

The notion of a based root datum is defined in \cite{Springer}.  Briefly, it is a quadruple $\Psi=(X, \Delta, X^{\vee}, \Delta^{\vee})$ where:

\begin{itemize}
\item $X$ and $X^{\vee}$ are free, finitely generated abelian groups, in duality by a pairing
\begin{equation*} 
\lip, \rip: X \times X^{\vee} \to \Z.
\end{equation*}
\item $\Delta$ (resp. $\Delta^{\vee}$) is a finite subset of $X$ (resp. of $X^{\vee}$).
\item There is a bijection $\alpha \mapsto \alpha^{\vee}$ from $\Delta$ onto $\Delta^{\vee}$.
\item For all $\alpha$, we have $\lip \alpha, \alpha^{\vee} \rip=2$.
\item If $s_{\alpha}$ (resp. $s_{\alpha^{\vee}}$) is the reflection of $X$ (resp., of $X^{\vee}$) determined by $\alpha$ and $\alpha^{\vee}$, then $s_{\alpha}(\Delta) \subset \Delta$ (resp., $s_{\alpha^{\vee}}(\Delta^{\vee}) \subset \Delta^{\vee}$).
\end{itemize}

The dual of  $\Psi=(X, \Delta, X^{\vee}, \Delta^{\vee})$ is given simply by $\Psi^{\vee}=(X^{\vee}, \Delta^{\vee}, X, \Delta)$.

Let $\Psi=(X, \Delta, X^{\vee}, \Delta^{\vee})$ and $\Psi'=(X', \Delta', X^{' \vee}, \Delta^{' \vee})$ be two root data.  Then an isomorphism between $\Psi$ and $\Psi'$ is a homomorphism $f:  X \to X'$ so that $f$  induces a bijection of $\Delta$ onto $\Delta'$ and so that the transpose of $f$ induces a bijection of $\Delta^{\vee}$ onto $\Delta^{' \vee}$.

Let $G$ be a connected reductive group over an algebraically closed field $F$.  Fix a maximal torus $T$ and a Borel subgroup $B$ of $G$ with $T \subseteq B$.
We say in this situation that $(T,B)$ is a pair (for $G$).  The choice of pair determines a based root datum 
\begin{equation*}
\Psi_0(G,T,B)=(X^*(T), \Delta(T,B), X_*(T), \Delta^{\vee}(T,B))
\end{equation*}
for $G$.  Here $\Delta(T,B)$ is the set of simple $B$-positive roots of $T$, and $\Delta^{\vee}(T,B)$ is the set of simple $B$-positive roots of $T$.  
If another pair $T' \subseteq B'$ is chosen, the new based root datum obtained is canonically isomorphic to the original via an inner automorphism $\alpha$ of $G$.  We have $\alpha(T')=T$ and $\alpha(B')=B$.  
Although the inner automorphism $\alpha$ need not be unique, its restriction to an isomorphism $T' \isom T$ is unique.

We may remove the dependence of the based root datum on the choice of pair as follows.  Write $X^*$ (resp. $\Delta$, $X_*$, and $\Delta^{\vee}$) for the inverse limit over the set of pairs $(T,B)$ of $X^*(T)$ (resp. $\Delta(T,B)$, $X_*(T)$, $\Delta^{\vee}(T,B)$).  Then we simply define the based root datum of $G$ to be
\begin{equation*}
\Psi_0(G)=(X^*,\Delta, X_*, \Delta^{\vee}).
\end{equation*} 

Conversely, given a based root datum, the group $G$ over $F$ is uniquely determined up to isomorphism.

Let $G$ be a connected reductive group over a field $F$, and $\Psi_0(G)$ a based root datum of $G_{\ol F}$.   Then $\Gamma_F$ acts naturally (via isomorphisms) on $\Psi_0(G)$.    The action of $\Gamma_F$ on $G$ is said to be an $L$-action if it fixes some splitting of $G$ (see Section 1.3 of \cite{K84}).

Conversely, given a based root datum with a $\Gamma_F$-action, the group $G$ over $F$ is uniquely determined up to isomorphism.

\begin{defn} A dual group for $G$ is the following data:

\begin{enumerate}
\item  A connected complex reductive group with a based root datum $\Psi_0(\hat{G})$.  We write its complex points as $\hat{G}$.
\item An $L$-action of $\Gamma_{F}$ on $\hat{G}$.
\item A $\Gamma_{F}$-isomorphism from $\Psi_0(\hat{G})$ to the dual of $\Psi_0(G)$.
 \end{enumerate}
\end{defn}

To specify the isomorphism for iii) above, one typically fixes pairs $(T_0,B_0)$ of $G$ and $(\hat S_0, \hat B_0)$ of a dual group $\hat G$ and an isomorphism from
$\Psi_0(\hat{G},\hat S_0, \hat B_0)$ to the dual of $\Psi_0(G,T_0,B_0)$.  

In the case that $G$ is a torus $T$, the dual group $\hat{T}$ is simply given by 
\begin{equation} \label{dual torus}
\hat{T}=X^*(T) \otimes_{\Z} \C^{\times},
\end{equation}
with the $\Gamma_F$-action induced from $X^*(T)$.  There are canonical $\Gamma_F$-isomorphisms $X^*(\hat T) \isom X_*(T)$ and $X_*(\hat T) \isom X^*(T)$. 

The formalism for dual groups encodes canonical isomorphisms between tori.  If $T$ and $T'$ are tori, and $\varphi: T \to T'$ is a homomorphism, it induces a homomorphism $\hat{T'} \to \hat{T}$ in the evident way.

Suppose that $(T,B)$ is a pair for $G$ and $(\hat S, \hat B)$ is a pair for $\hat{G}$.  By iii) above, one has an in particular a fixed isomorphism from $\Psi_0(G,T,B)$ to the dual of $\Psi_0(\hat G, \hat S, \hat B)$.  In particular this yields an isomorphism from $X^*(T)$ to $X_*(\hat S)$, which induces an isomorphism 
\begin{equation} \label{determined}
\hat T \isom \hat S.
\end{equation}

Next, let $G$ be a connected reductive group over a field $F$, which is either local or global.  

\begin{defn}
An endoscopic group for $G$ is a triple $(H,s, \eta)$, where
\begin{itemize}
\item $H$ is a quasi-split connected group, with a fixed dual group $\hat{H}$ as above.
\item $s \in Z(\hat{H})$.
\item $\eta: \hat{H} \to \hat{G}$ is an embedding.
\item The image of $\eta$ is $(\hat{G})_{\eta(s)}^{\circ}$, the connected component of the centralizer in $\hat{G}$ of $\eta(s)$.
\item The $\hat{G}$-conjugacy class of $\eta$ is fixed by $\Gamma_F$.

Cohomology of $\Gamma_F$-modules then yields a boundary map
\begin{equation*}
\left[ Z(\hat{H})/ Z(\hat{G}) \right]^{\Gamma_F} \to H^1(F,Z(\hat{G})).
\end{equation*}

\item The image of $s$ in $Z(\hat{H})/Z(\hat{G})$ is fixed by $\Gamma$, and its image under the above boundary map is trivial if $F$ is local and locally trivial if $F$ is global.
\end{itemize}
 
 An endoscopic group is elliptic if the identity components of $Z(\hat{G})^{\Gamma_F}$ and $Z(\hat{H})^{\Gamma_F}$ agree.
\end{defn}
 
The notion of an isomorphism of endoscopic groups is defined in Section 7.5 of \cite{K84}; we do not review it here.

\subsection{Langlands Correspondence}  \label{Langlands packets}
 Let $G$ be a connected reductive group over $\R$.  
In this section we review elliptic Langlands parameters for $G$ and the corresponding $L$-packets for discrete series representations of $G(\R)$.  Our main references are  \cite{Borel} and \cite{K90}.  Write $W_\R$ for the Weil group of $\R$, and $W_{\C}$ for the canonical image of $\C^{\times}$ in $W_{\R}$.  There is an exact sequence
\begin{equation*}
 1 \to W_{\C} \to W_\R \to \Gamma_\R \to 1. 
\end{equation*}
The Weil group $W_\R$ is generated by $W_{\C}$ and a fixed element $\tau$ satisfying $\tau^2=-1$
and $\tau z \tau^{-1}=\ol{z}$ for $z \in W_{\C}$.   The action of $\Gamma_{\R}$ on $\hat{G}$ inflates to an action of $W_{\R}$ on $\hat{G}$, and through this action we form the $L$-group ${}^LG=\hat{G} \rtimes W_\R$.

A Langlands parameter $\varphi$ for $G$ is an equivalence class of continuous homomorphism $\varphi: W_\R \to {}^LG$ commuting with projection to $\Gamma_{\R}$, satisfying a mild hypothesis on the image (see \cite{Borel}).  The equivalence relation is via inner automorphisms from $\hat{G}$. One associates to a Langlands parameter $\varphi$ an $L$-packet $\Pi(\varphi)$ of irreducible admissible representations of $G$.

Suppose that $G$ is cuspidal, so that there is a discrete series representation of $G(\R)$.  This implies that the longest element $w_0$ of the Weyl group $\Omega$ acts as $-1$ on $X_*(T)$.  If $\varphi$ is a Langlands parameter, write $C_{\varphi}$ for the centralizer of $\varphi(W_{\R})$ in $\hat{G}$ and $\hat{S}$ for the centralizer of $\varphi(W_{\C})$ in $\hat{G}$.  Write $S_{\varphi}$ for the product $C_{\varphi}Z(\hat{G})$.   We say $\varphi$ is elliptic if $S_\varphi/Z(\hat{G})$ is finite, and describe the $L$-packet $\Pi(\varphi)$ in this case.

Since $\varphi$ is elliptic, the centralizer $\hat{S}$ is a maximal torus in $\hat{G}$.  Since $\varphi$ commutes with the projection to $\Gamma_{\R}$, it restricts to a homomorphism
\begin{equation*}
W_{\C} \to \hat{S} \times \{ 1 \}.
\end{equation*} 
We may view this restriction as a continuous homomorphism $\varphi: \C^{\times} \to \hat{S}$, which may be written in exponential form 
\begin{equation*}
\varphi(z)=z^{\mu} \ol{z}^{\nu}
\end{equation*}
with $\mu$ and $\nu$ regular elements of $X_*(\hat{T})_{\C}$.   Write $\hat{B}$ for the unique Borel subgroup of $\hat{G}$ containing $\hat{S}$ so that $\lip \mu, \alpha \rip$ is positive for every root $\alpha$ of $\hat{S}$ that is positive for $\hat{B}$.  We say that $\varphi$ determines the pair $(\hat{S}, \hat{B})$, at least up to conjugacy in $\hat{G}$.

Let $B$ be a Borel subgroup of $G_\C$ containing $T$.  Then $\varphi$ and $B$ determine a quasi-character $\chi_B=\chi(\varphi,B)$, as follows.  There is a canonical (up to $\hat{G}$-conjugacy) homomorphism $\eta_B:  {}^LT \to {}^LG$ described in \cite{K90} so that for $z \in W_{\C}$,
\begin{equation*}
\eta_B(z)=z^\rho \ol{z}^{-\rho} \times z \in \hat{G} \rtimes W_\R.
\end{equation*}
Here $\rho=\rho_G$ is the half sum of the $B$-positive roots for $T$.  Then a Langlands parameter $\varphi_B$ for $T$ may be chosen so that $\varphi=\eta_B \circ \varphi_B$.
Finally $\chi_B$ is the quasi-character associated to $\varphi_B$ by the Langlands correspondence for $T$ (as described in Section 9.4 of \cite{Borel}).
 
Write $\BB$ for the set of Borels of $G_\C$ containing $T$.  The $L$-packet associated to $\varphi$ is indexed by  $\Omega_\R \backslash
\BB$.  For $B \in \Omega_\R \backslash \BB$, a representation $\pi(\varphi,B)$ in the $L$-packet is given by the irreducible discrete series representation of $G(\R)$ whose character
$\Theta_\pi$ is given on regular elements $\gm$ of $T(\R)$ by

\[ (-1)^{q(G)} \sum_{\omega \in \Omega_\R} \chi_{\omega(B)}(\gm) \cdot \Delta_{\omega(B)}(\gm)^{-1}. \]

Here $\Delta_B$ is the usual discriminant
\[ \Delta_B(\gm)= \prod_{\alpha >0 \text{ for $B$}} (1-\alpha(\gm)^{-1}). \]  
    
Finally, let
\begin{equation*}
\Pi(\varphi)= \{ \pi(\varphi,B) \mid B \in \Omega_\R \backslash
\BB \}.
\end{equation*}

It has order $d(G)=|\Omega  / \Omega_\R|$.  There is a unique irreducible finite-dimensional complex
representation $E$ of $G(\C)$ with the same infinitesimal character and central character as the representations in this $L$-packet.  It has highest weight $\mu-\rho \in X^*(T)$
with respect to $B$.  The isomorphism classes of such $E$ are in one-to-one correspondence with elliptic Langlands parameters $\varphi$, and  we often write $\Pi_E$ for $\Pi(\varphi)$.

\begin{defn} We say that a discrete series representation $\pi \in \Pi_E$ is regular if the highest weight of $E$ is regular.
\end{defn}

\subsection{Measures and Orbital Integrals} \label{algorithm}
 
Let $G$ be a locally compact group with Haar measure $dg$.  If $f$ is a continuous function on $G$, write $fdg$ for the measure on $G$ given by
\begin{equation*}
\varphi \mapsto \int_G \varphi(g)f(g)dg,
\end{equation*} 
for $\varphi$ continuous and compactly supported in $G$.
We will refer to the measures obtained in this way simply as ``measures".  If $G$ is a $p$-adic, real, or adelic  Lie group, we require that $f$ be suitably smooth.

In this paper, we will view orbital integrals and Fourier transforms as distributions defined on measures, rather than on functions.  This approach eases their dependence on choices of local measures, choices which do not matter in the end.

If $K$ is an open compact subset of $G$, then write $e_K$ for the measure given by $fdg$, where 
$f$ is the characteristic function of $K$ divided by $\vol_{dg}(K)$.  Note that the measure $e_K$ is independent of the choice of Haar measure $dg$.

Let $G$ be a reductive group defined over a local field $F$.
Fix a Haar measure $dg$ on $G(F)$. Let $fdg$ be a measure on $G(F)$, and take a semisimple element $\gamma \in G(F)$.   Fix a Haar measure $dt$ of $G(F)_{\gm}^{\circ}$.  Then we write $O_\gm(f dg; dt)$ for the usual orbital integral

\begin{equation*}
O_\gm(f dg; dt)=\int_{G_\gm^{\circ}(F) \bks G(F)} f(g^{-1} \gm g) \frac{dg}{dt}.
\end{equation*}

Many cases of finite orbital integrals are easy to compute by the following result, a special case extracted from Section 7 of \cite{K86}:

\begin{prop} \label{orbitalintegrals} Let $F$ be a $p$-adic field with ring of integers $\OO$.  Let $G$ be a split connected reductive group defined over $\OO$, and $K=G(\OO)$.  Suppose that $\gm \in K$ is semisimple, and that $1-\alpha(\gm)$ is either $0$ or a unit for every root $\alpha$ of $G$.  Let $\gm'$ be stably conjugate to $\gm$.  Then $O_{\gm'}(e_{K}; dt)$ vanishes unless $\gm'$ is conjugate to $\gm$, in which case

\begin{equation*}
O_{\gm'}(e_{K}; dt)=\vol_{dt}(G_\gm^{\circ}(F) \cap K)^{-1}.
\end{equation*}
\end{prop}

Now let $G$ be a reductive group defined over $\Q$.  

Let $f^\infty dg_f$ be a measure on $G(\Aff_f)$ and take a semisimple element $\gm \in G(\Aff_f)$.  Fix a Haar measure  $dt_f$ of $G_{\gm}^{\circ}(\Aff_f)$.   Write $O_\gm(f^\infty dg_f; dt_f)$ for the orbital integral

\begin{equation*}
O_\gm(f^\infty dg_f; dt_f)=\int_{G_\gm^{\circ}(\Aff_f) \bks G(\Aff_f)} f^{\infty}(g^{-1} \gm g) \frac{dg_f}{dt_f}.
\end{equation*}
 
We also have the stable orbital integrals
\begin{equation*}
SO_\gm(f^\infty dg_f; dt_f)= \sum_i e(\gm_i) O_{\gm_i}(f^\infty dg_f; dt_{i,f}),
\end{equation*}
the sum being over $\gm_i \in G(\Aff_f)$ (up to $G(\Aff_f)$-conjugacy) whose local components are stably conjugate to $\gm$.  The centralizers of $\gm$ and a given $\gm_i$ are inner forms of each other, and we use corresponding measures $dt_f$ and $dt_{i,f}$.
The number $e(\gm_i)$ is defined as follows:  Recall that for a reductive group $A$ over a local field, Kottwitz has defined an invariant $e(A)$ in \cite{Kottwitz sign}.  It is equal to $1$ if $A$ is quasi-split.
For each place $v$ of $\Q$, write $\gm_{i,v}$ for the $v$th component of $\gm_i$.  Let
\begin{equation*}
e(\gm_{i,v})=e(G_{\gm_{i,v}}^{\circ}(\Q_v)).
\end{equation*} 
Finally, let
\begin{equation*}
e(\gm_i)=\prod_v e(\gm_{i,v}).
\end{equation*}
 
\begin{defn} Let $M$ be a Levi component of a parabolic subgroup $P$ of $G$, and $dm_f$ a Haar measure on $M(\Aff_f)$.   Given a measure $f^\infty dg_f$, its ``$M$-constant term'' is the measure $f_M^{\infty} dm_f$, where $f_M^{\infty}$ is defined via
\begin{equation*}
 f_M^\infty(m)=  \delta_{P(\Aff_f)}^{-\half}(m) \int_{N(\Aff_f)} \int_{K_f} f^\infty(k^{-1}nmk)dk_fdn_f. 
 \end{equation*}
 Here we fix the Haar measure $dk_f$ on $K_f$ giving it mass one, and the Haar measure $dn_f$ on $N(\Aff_f)$ is chosen so that $dg_f=dk_f dn_f dm_f$.  The function $\delta_{P(\Aff_f)}$ is the modulus function on $P(\Aff_f)$.
\end{defn} 
 
It is independent of the choice of parabolic subgroup $P$.
 
\begin{prop} Let $G$ be a split group defined over $\Z$ and let $K_f=G(\OO_f)$.  Then
\begin{equation*}
(e_{K_f})_M= e_{M(\Aff_f) \cap K_f}.
\end{equation*}
\end{prop}

\begin{proof}
Write $e_{K_f}=f^{\infty} dg_f$.  Then it is easy to see that $f_M^{\infty}(m)=0$ unless $m \in K_f$.  If $m \in K_f$, we compute that
\begin{equation*}
f_M^{\infty}(m)=\frac{\vol_{dk_f}(K_f) \vol_{dn_f}(K_f \cap N(\Aff_f))}{\vol_{dg_f}(K_f)}.
\end{equation*}
The result follows since $\vol_{dg_f}(K_f)=\vol_{dm_f}(M(\Aff_f) \cap K_f) \vol_{dn_f}(N(\Aff_f) \cap K_f) \vol_{dk_f}(K_f)$.
\end{proof}

\subsection{Pseudocoefficients} \label{PC}

We continue with a connected reductive group $G$ over $\Q$, and adopt some terminology from \cite{A-L^2}.
Fix a maximal compact subgroup $K_{\R}$ of $G(\R)$.  We put $K_{\R}'=K_{\R}A_G(\R)^+$.  Given a quasicharacter (smooth homomorphism to $\C^{\times}$) $\xi$ on $A_G(\R)^+$, write $\mathcal{H}_{\ac}(G(\R),\xi)$ for the set of smooth, $K_{\R}'$-finite functions on $G(\R)$ which are compactly supported modulo $A_G(\R)^+$, and transform under $A_G(\R)^+$ according to $\xi$.  Write $\Pi(G(\R),\xi)$ for the set of irreducible representations of $G(\R)$ whose central character restricted to $A_G(\R)^+$ is equal to $\xi$.

Given a function $f \in \mathcal{H}_{\ac}(G(\R),\xi^{-1})$, a representation $\pi \in \Pi(G(\R),\xi)$, and a Haar measure $dg_{\infty}$ on $G(\R)$, write $\pi(f dg_{\infty})$ for the operator on the space of $\pi$ given by the formula:

\begin{equation*}
\pi(f dg_{\infty})= \int_{G(\R)/A_G(\R)^+} f(x) \pi(x) dg_{\infty}.
\end{equation*}
Here we give $A_G(\R)^+$ the measure corresponding to Lebesgue measure on $\R^n$, if $A_G$ is $n$-dimensional.  The operator is of trace class.  

Write $\Pi_{\tmp}(G(\R),\xi)$ (resp. $\Pi_{\disc}(G(\R),\xi))$ for the subset of tempered (resp. discrete series) representations in $\Pi(G(\R),\xi)$.

\begin{defn} Suppose that $f \in \mathcal{H}_{\ac}(G(\R),\xi^{-1})$.  Then we say that the measure $fdg_{\infty}$ is cuspidal if $\tr \pi(f dg_{\infty})$, viewed as a function on  $\Pi_{\tmp}(G(\R),\xi)$, is supported on $\Pi_{\disc}(G(\R),\xi)$. \end{defn}

Write $\tilde{E}$ for the contragredient of the representation $E$.  In \cite{A-L^2}, Arthur employs cuspidal measures $f_E \in \mathcal{H}_{\ac}(G(\R),\xi^{-1})$ whose defining property is that, for all $\pi \in \Pi_{\tmp}(G(\R),\xi)$,
\begin{equation} \label{packet ps}
\tr \pi(f_E dg_{\infty})=
  \begin{cases} & (-1)^{q(G)} \text{ if } \pi \in \Pi_{\tilde{E}}, \\
              & 0, \text{ otherwise}. \end{cases}
\end{equation}
 
Such measures can be broken down further.  

\begin{defn} \label{pseudoscience}
Fix a representation $\pi_0 \in \Pi_{\disc}(G(\R), \xi^{-1})$, and let $f_0 \in \mathcal{H}_{\ac}(G(\R),\xi^{-1})$.
Suppose that the measure $f_0 dg_{\infty}$  satisfies, for all $\pi \in \Pi_{\tmp}(G(\R),\xi)$,
\begin{equation*}
\tr \pi(f_0 dg_{\infty})=
  \begin{cases} & (-1)^{q(G)} \text{ if } \pi \cong \tilde \pi_0, \\
              & 0, \text{ otherwise}. \end{cases}
\end{equation*}
  
It follows from the corollary in Section 5.2 of \cite{CD} that such functions exist.   Pick such a function $f_0$, and put
\begin{equation*}
e_{\pi_0}=f_0 dg_{\infty}.
\end{equation*}
\end{defn} 
 
 Suppose that for each $\pi \in \Pi_E$ we fix measures $e_{\pi}$ as above.  Let
\begin{equation*}
f_E dg_{\infty} =\sum_{\pi} e_{\pi},
\end{equation*}
the sum being over $\pi \in \Pi_E$.  Then clearly $f_E dg_{\infty}$ satisfies Arthur's condition (\ref{packet ps}).

{\bf Remark:} The measure $(-1)^{q(G)}e_{\pi}$ is called a pseudocoefficient of $\tilde \pi$. 

\section{Transfer} \label{Transfer}

We sketch the important theory of transfer in the form that we will use in this paper.

Suppose that $G$ is a real connected reductive group, and that $(H,s,\eta)$ is an elliptic endoscopic group for $G$.  Fix an elliptic maximal torus $T_H$ of $H$, an elliptic maximal torus $T$ of $G$, and an isomorphism 
$j: T_H \isom T$ between them.  Also fix a Borel subgroup $B$ of $G_\C$ containing $T$ and a Borel subgroup $B_H$ of $H_\C$ containing $T_H$.

Suppose that $\xi$ is a quasicharacter on $A_G(\R)$, and that $f_{\infty} \in \mathcal{H}_{\ac}(G(\R),\xi^{-1})$, with $f_{\infty} dg_{\infty}$ cuspidal.
There is a corresponding quasicharacter $\xi_H$ on $A_H(\R)$ described in Section 5.5 of \cite{SVAF}.  

There is also a measure $f_{\infty}^H dh_{\infty}$ on $H(\R)$ with $f_{\infty}^H \in \mathcal{H}_{\ac}(H(\R),\xi_H^{-1})$, having matching character values.   (See \cite{Shelstad}, \cite{CD}, \cite{CD2}, \cite{LS 90}.)
More specifically, let $\varphi_H$ be a tempered Langlands parameter for $H_{\R}$, and write $\Pi_H=\Pi(\varphi_H)$ for the corresponding $L$-packet of discrete series representations of $H(\R)$.  Transport $\varphi_H$ via $\eta$ to a tempered Langlands parameter $\varphi_G$ for $G$.  The parameters $\varphi_G$ and $\varphi_H$ determine pairs $(\hat S, \hat B)$ and $(\hat S_H, \hat B_H)$ as in Section \ref{Langlands packets}.

Then
\begin{equation} \label{charidentity}
\tr \Pi_H(f_{\infty}^H dh_{\infty})= \sum_{\pi \in \Pi} \Delta_{\infty}(\varphi_H,\pi) \cdot \tr \pi(f_{\infty} dg_{\infty}),
\end{equation}
using Shelstad's transfer factors $\Delta_{\infty}(\varphi_H,\pi)$.   Both sides of (\ref{charidentity}) vanish unless $\Pi_H$ is a discrete series packet.  In particular, $f_{\infty}^H dh_{\infty}$ is cuspidal, and it may be characterized by (\ref{charidentity}).  (The transfer $f_{\infty}^H dh_{\infty}$ is only defined up to the kernel of stable distributions.)  We may use this formula to identify it as a combination of pseudocoefficients.
 
It is a delicate matter to specify the transfer factors.  We will use a formula for $\Delta_{\infty}(\varphi_H,\pi)$  from \cite{K90}, which is itself a reformulation of \cite{Shelstad}.  One must carefully specify the duality between $G$ and $\hat{G}$, and between $H$ and $\hat{H}$, because this factor depends on precisely how this is done.  It also depends on the isomorphism $j: T_H \isom T$, which must be compatible with correspondences of tori determined by the Langlands parameters, as specified below.   

\begin{defn} The triple $(j,B_T, B_{T_H})$ is aligned with $\varphi_H$ if the following diagram commutes:

\begin{equation} \label{alignment}
\begin{CD}
\hat{T} @ >>> \hat{S} \\
@V \hat{j} V V        @ AA \eta A \\
\hat{T}_H @ >>> \hat{S}_H
\end{CD}.
\end{equation}

Here the isomorphism $\hat{T} \to \hat{S}$ (resp., $\hat{T}_H \to \hat{S}_H$) is determined, as in (\ref{determined}), by $(B,\hat{B})$ (resp., $(B_H,\hat{B}_H)$).  The map $\hat{j}$ is the dual map to $j$ using the identification (\ref{dual torus}) of the dual tori.
\end{defn}

For each $\omega \in \Omega$,  there is a character 
\begin{equation*}
a_{\omega}: \left( \hat{T} / Z(\hat{G}) \right)^{\Gamma_{\R}} \to \{ \pm 1 \}
\end{equation*}
described in \cite{K90}.

If the triple $(j,B_T,B_{T_H})$ is aligned with $\varphi_H$, then we may take as transfer factors
\begin{equation*}
\Delta_{\infty}(\varphi_H, \pi(\varphi, \omega^{-1}(B)))= \lip a_{\omega}, \hat j^{-1}(s) \rip.
\end{equation*}

Next, let $G$ be a connected reductive algebraic group over $\Q$, and $(H,s,\eta)$ an endoscopic group for $G$.  Given a measure $f^{\infty}dg_f$ on $G(\Aff_f)$, there is a measure $f^{\infty H} dh_f$ on $H(\Aff_f)$ so that for all $\gm_H \in H(\Aff_f)$ suitably regular, one has
\begin{equation*}
SO_{\gm_H}(f^{\infty H} dh_f )= \sum_{\gm} \Delta^{\infty}(\gm_H,\gm) O_{\gm}(f^{\infty} dg_f).
\end{equation*}
The sum is taken over $G(\Aff_f)$-conjugacy classes of ``images" $\gm \in G(\Aff_f)$ of $\gm_H$.   We have written $\Delta^{\infty}(\gm_H,\gm)$ for the the Langlands-Shelstad transfer factors.   One takes matching measures on the centralizers of $\gm_H$ and the various $\gm$ in forming the quotient measures for the orbital integrals.  We have left out many details; please see Langlands-Shelstad \cite{LS 90} and Kottwitz-Shelstad \cite{KS} for definitions, and Ng\^{o} \cite{Ngo} for the celebrated proof.

\section{Arthur's $\Phi$-Function} \label{APF}

In this section we consider a reductive group $G$ defined over $\R$.  Let $T$ be a maximal torus contained in a Borel subgroup $B$ of $G_{\C}$. Let $A$ be the split part of $T$, let $T_e$ be the
maximal elliptic subtorus of $T$, and $M$ the centralizer of $A$ in $G$. It is a Levi subgroup of $G$ containing $T$. Let $E$ be an irreducible
finite-dimensional representation of $G(\C)$, and consider the $L$-packet $\Pi_E$ of discrete series representations $\pi$ of $G(\R)$ which have the
same infinitesimal and central characters as $E$. Write $\Theta_\pi$ for the character of $\pi$, and put

\[ \Theta^E= (-1)^{q(G)} \sum_{\pi \in \Pi_E} \Theta_\pi. \]

Note that $\Theta^E(\gm)$ will not extend continuously to all elements $\gm \in
T(\R)$, in particular to $\gm= 1$. Define the function $D^G_M$ on $T$ by
\[ D^G_M(\gm) = \det(1-\Ad(\gm);\Lie(G)/\Lie(M)). \]

Then a result of Arthur and Shelstad \cite{A-L^2}  states that the function
\[ \gm \mapsto |D^G_M(\gm)|^\half \Theta^E(\gm), \]
defined on the set of regular elements $T_{\reg} (\R)$, extends continuously to $T(\R)$.  We denote this extension by $\Phi_M(\gm,\Theta^E)$.
 The following closed expression for $\Phi_M(\gm,\Theta^E)$ when $\gm \in T_e$ is given in \cite{SpalPhi}.

\begin{prop}  \label{MPI} If $\gm \in T_e(\R)$, then
\begin{equation} \label{SpalPhiformula}
 \Phi_M(\gm, \Theta^E)=(-1)^{q(L)}|\Omega_L| \sum_{\omega \in \Omega^{LM}} \eps(\omega) \tr(\gm; V^M_{\omega(\lambda_B+\rho_B)-\rho_B}).
\end{equation}

In particular,
\begin{enumerate}
\item If $T$ is elliptic then $M=G$ and $\Phi_G(\gm,\Theta^E)= \tr(\gm; E)$.

\item If $T$ is split and $z \in A_G(\R)$ then $M=A$ and $\Phi_A(z,\Theta^E)=(-1)^{q(G)} |\Omega_G| \lambda_0(z)$.
\end{enumerate}
 \end{prop}

The notation needs to be explained.  Here $L$ is the centralizer of $T_e$ in $G$.  The roots of $T$ in $L$ (resp. $M$) are the real (resp. imaginary) roots of $T$ in $G$.  Write $\Omega_L$ and $\Omega_M$ for the respective Weyl groups.  Write $\Omega^{LM}$ for the set of elements which are simultaneously Kostant representatives for both $L$ and $M$, relative to $B$.  We write $\eps$ for the sign character of $\Omega_G$.  Finally by
$V^M_{\omega(\lambda_B+\rho_B)-\rho_B}$ we denote the irreducible finite-dimensional representation of $M(\C)$ with highest weight $\omega(\lambda_B+\rho_B)-\rho_B$, where $\lambda_B$ is the $B$-dominant highest weight of $E$.
Finally, $\lambda_0$ is the character by which $A_G(\R)$ acts on $E$.

For the case of central $\gm=z$, computing $\Phi_M(z,\Theta^{E})$ amounts to computing the dimensions of finite-dimensional representations of $M(\C)$ with various highest weights.  For this we use the Weyl dimension formula, in the following form:

\begin{prop} \label{Weyl} (Weyl Dimension Formula) Let $G$ be a complex reductive group, $T$ a maximal torus in $G$, contained in a Borel subgroup $B$.  Write $\rho_B$ for the half-sum of the positive roots for $T$ in $G$ (with respect to $B$).  Let $\lambda_B \in X^*(T)$ be a positive weight.  Then there is a unique irreducible representation $V_{\lambda_B}$ of $G$ with highest weight $\lambda_B$.  Its dimension is given by
\begin{equation*}
\dim_\C V_{\lambda_B}= \prod_{\alpha >0} \frac{ \lip \alpha, \lambda_B+ \rho_B \rip}{\lip \alpha, \rho_B \rip}.
\end{equation*}
\end{prop}

Here $\lip, \rip$ is a $\Omega_G$-invariant inner product on $X^*(T)_\R$, which is unique up to a scalar.

\section{Kottwitz's Formula} \label{Kottwitz's Formula}

\subsection{Various Invariants} \label{Various Invariants}
 
In this section we introduce some invariants involved in Kottwitz's formula.  
 
By $\ol{G}$ we generally denote an inner form of $G_{\R}$ such that $\ol{G}/A_G$ is anisotropic over
$\R$.

\begin{defn} Let $G$ be a cuspidal reductive group over $\R$, and $dg_{\infty}$ a Haar measure on $G(\R)$.  Let
\begin{equation*}
\ol{v}(G; dg_{\infty})=e(\ol{G}) \vol(\ol{G}(\R)/ A_G(\R)^+).
\end{equation*}
\end{defn}

This is a stable version of the constant $v(G)$ which appears in \cite{A-L^2}.

Here $e(\ol{G})$ is the sign defined in \cite{Kottwitz sign}.  (Note that $e(\ol{G})=(-1)^{q(G)}$ when $G$ is quasisplit.)
In both cases the Haar measure on $\ol{G}(\R)$ is transported from $dg_\infty$ on $G(\R)$ in the usual way, and the measure on $A_G(\R)^+$ is the standard Lebesgue measure.

\begin{defn} Let $G$ be a cuspidal connected reductive group over $\Q$.  Then $G$ contains a maximal torus $T$ so that $T/ A_G$ is anisotropic over $\R$.  Write  $T_{\simp}$ for the inverse image in $G_{\simp}$ of $T$.
  Then $k(G)$ is the cardinality of the
set $\im[ H^1(\R,T_{\simp}) \to H^1(\R,T)]$.
\end{defn}
 
\begin{defn} If $G$ is a reductive group over $\Q$, write $\tau(G)$ for the Tamagawa number of $G$, as defined in \cite{Ono}.
\end{defn}

By \cite{K88} or \cite{SVAF}, the Tamagawa numbers $\tau(G)$ for a reductive group $G$ over $\Q$ may be computed using the formula
\begin{equation*}
\tau(G)=|\pi_0(Z(\hat{G})^{\Gamma_\Q})| \cdot |\ker^1(\Q,Z(\hat{G}))|^{-1}.
\end{equation*}

Here $\pi_0$ denotes the topological connected component.

\begin{defn} Let $M$ be a Levi subgroup of $G$.  Then put
\begin{equation*}
n^G_M=[N_G(M)(\Q) : M(\Q)].
\end{equation*}
\end{defn}

Here $N_G(M)$ denotes the normalizer of $M$ in $G$.

\begin{defn} Let $\gm \in M(\Q)$ be semisimple.  Then put 
\begin{equation*}
\ol{\iota}^M(\gm)=|(M_{\gm} / M_{\gm}^{ \circ})(\Q)|
\end{equation*}
and
\begin{equation*}
\iota^M(\gm)=[M_{\gm}(\Q) : M_{\gm}^{ \circ}(\Q)].
\end{equation*}

\end{defn}

Let $(H,s,\eta)$ be an endoscopic triple for $G$, and write $\Out(H,s,\eta)$ for its outer automorphisms.  Put
\begin{equation*}
\iota(G,H)= \tau(G) \tau(H)^{-1} |\Out(H,s,\eta)|^{-1}.
\end{equation*}

\subsection{The Formula}

In this section we give Kottwitz's formula from \cite{SVAF}.

Our $G$ will now be a cuspidal connected reductive group over $\Q$.  Let $f^{\infty} \in C_c^{\infty}(G(\Aff_f))$ and $f_{\infty}   \in \mathcal{H}_{\ac}(G(\R),\xi)$ for some $\xi$.  We consider measures $fdg$ of the form $fdg=f^{\infty}dg_f \cdot f_{\infty}dg_{\infty} \in C_c^{\infty}(G(\Aff))$, for some decomposition $dg=dg_f dg_{\infty}$ of the Tamagawa measure on $G(\Aff_f)$.  Also choose such decompositions for every cuspidal Levi subgroup $M$ of $G$.

First we define the stable distribution $S\Phi_M$ at the archimedean place:

\begin{defn} Let  $M$ be a cuspidal Levi subgroup of $G$.  Let $\gm \in M(\Q)$ be elliptic, and pick a Haar measure $dt_{\infty}$ of $M_{\gm}^{\circ}(\R)$.  Then $S \Phi_M(\gm,f_\infty dg_\infty; dt_{\infty})$
is defined to be
\[ (-1)^{\dim(A_M/A_G)} k(M)k(G)^{-1} \ol{v}(M_\gm^{\circ}; dt_{\infty})^{-1} \sum_\Pi \Phi_M(\gm^{-1},\Theta_\Pi) \tr \Pi(f_\infty dg_\infty), \]
\end{defn}
the sum being taken over $L$-packets of discrete series representations.

Here is the basic building block of Kottwitz's formula:
\begin{defn}
Let $M$ be a cuspidal Levi subgroup of $G$, and $\gm \in M(\Q)$ an elliptic element.  Pick Haar measures $dt_f$ on $M_{\gm}^{\circ}(\Aff_f)$ and $dt_{\infty}$ on $M_{\gm}^{\circ}(\R)$ whose product is the Tamagawa measure $dt$ on $M_{\gm}^{\circ}(\Aff)$.

We define
\begin{equation*}
ST_g(fdg,\gm, M)=  (n^G_M)^{-1} \tau(M) \ol{\iota}^M(\gm)^{-1} SO_\gm(f_M^\infty dm_f; dt_f) S \Phi_M(\gm,f_\infty dg_\infty; dt_{\infty}).
\end{equation*}
\end{defn}

Here $f_M^{\infty}dm_f$ is the $M$-constant term of $f^{\infty}dg_f$.  The product $SO_{\gm}(f_M^{\infty}dm_f; dt_f) \ol v(M; dt_{\infty})$ is independent of the decompositions of $dt$ and $dg$.  We will therefore often write this simply as $SO_{\gm}(f_M^{\infty} dm_f) \ol v(M)$.  Similarly for other such products.

Kottwitz defines
\begin{equation*}
ST_g(fdg)= \sum_M \sum_{\gm \in M} ST_g(fdg,\gm, M).
\end{equation*}

Here $M$ runs over $G(\Q)$-conjugacy classes of cuspidal Levi subgroups in $G$, and the second sum runs over stable $M(\Q)$-conjugacy classes of semisimple elements $\gm \in M(\Q)$ which are elliptic in $M(\R)$.

For convenience we also define, for $\gm \in G(\Q)$ semisimple,
\begin{equation*}
ST_g(f dg, \gm)= \sum_M ST_g( fdg, \gm, M),
\end{equation*}
The sum being taken over cuspidal Levi subgroups of $G$ with $\gm \in M(\Q)$.

Kottwitz's stable version of Arthur's trace formula is given by
\begin{equation*}
\Kot(f dg)=\sum_{(H,s,\eta) \in \funE_0} \iota(G,H) ST_g(f^H dh),
\end{equation*}
where $\funE_0$ is the set of (equivalence classes of) elliptic endoscopic groups for $G$.

We record here the simpler form of $ST_g(fdg, \gm, M)$ when $\gm=z$ is in the rational points $Z(\Q)$ of the center of $G$.  
We have  
\begin{equation*}
ST_g(f dg, z, M)=     (-1)^{\dim(A_M/A_G)} \frac{k(M)}{k(G)} (n^G_M)^{-1} \tau(M) f_M^\infty(z)  \ol{v}(M; dm_{\infty})^{-1} \Phi_M(z^{-1}, \Theta_{\Pi}).
\end{equation*}

\subsection{Conjecture} \label{Conjecture}
 
Recall the stable cuspidal measure $f_E dg_{\infty}$ from Section \ref{PC}.   Fix any test function $f^{\infty} dg_f$ and put $f=f^{\infty}f_E dg$.

Let
\begin{equation*}
T_g(fdg)= \sum_M (n^G_M)^{-1} \sum_\gm \iota^M(\gm)^{-1} \tau(M_\gm) O_\gm(f_M^\infty dm_f) \Phi_M(\gm,f_E dg_{\infty}).
\end{equation*}

Here as in \cite{A-L^2}, $\Phi_M(\gm,-)$ is the unnormalized form of Arthur's distribution $I_M$ defined in  \cite{ITF1}.   

Now suppose that $\pi \in \Pi_{\disc}(G(\R), \xi)$, and let $K_f$ be an open compact subgroup of $G(\Aff_f)$.  Write
\begin{equation*}
L^2(G(\Q) \bks G(\Aff) / K_f, \xi)
\end{equation*}
for the space of functions on this double coset space which transform by $A_G(\R)^+$ according to $\xi$ and are square integrable modulo center.
Write $R_{\disc}(\pi, K_f)$ for the $\pi$-isotypical subspace of $L^2(G(\Q) \bks G(\Aff) / K_f, \xi)$; it is finite-dimensional.  If $f^{\infty}dg_f$ is $K_f$-biinvariant, then convolution gives an operator 
$R_{\disc}(\pi,f^{\infty}dg_f)$ on $R_{\disc}(\pi, K_f)$.  According to Arthur's Corollary 6.2 of \cite{A-L^2}, if the highest weight of $E$ is regular, then

\begin{equation*}
\sum_{\pi \in \Pi_E} \tr R_{\disc}(\pi,f^{\infty}dg_f) = T_g(fdg).
\end{equation*}

The main result of Kottwitz \cite{SVAF} is that when $f_{\infty} dg_{\infty}$ is stable cuspidal, then
\begin{equation*}
T_g(f dg)=\Kot(f dg).
\end{equation*}

Since we may assume $f_Edg_{\infty}= \sum_{\pi \in \Pi_E}  e_{\pi}$, the following conjecture is plausible:

\begin{conj} 
Fix a regular discrete series representation $\pi$ of $G(\R)$.  Let $f_{\infty} dg_{\infty}=e_{\pi}$ as in Section \ref{PC}.   Pick a measure $f^{\infty} dg_f $ with $f^{\infty} \in C_c(G(\Aff_f))$, and $dg_f dg_{\infty}=dg$ the Tamagawa measure on $G(\Aff)$.  Put $f=f^{\infty} f_{\infty}$.  Then

\begin{equation*}
\Kot(f dg)=\tr R_{\disc}(\pi,f^{\infty} dg_f).
\end{equation*}

\end{conj}

In particular, if we pick a compact open subgroup $K_f$ of $G(\Aff_f)$, and put  $f^\infty dg_f=e_{K_f}$, we obtain
\begin{equation*}
m_{\disc}(\pi,K_f)=\Kot(e_{\pi} e_{K_f}).
\end{equation*}

In this paper we give some evidence for this conjecture.  Moreover, we will see that $\Kot(fdg)$ is given by a closed algebraic expression, which is straightforward to evaluate, so 
long as one can compute the transfers $e_{\pi}^H$ at the real place, and evaluate the elliptic orbital integrals of $f^{\infty H} dh_f$ at the finite adeles.

 \section{Euler Characteristics} \label{Euler Characteristics}
We have finished our discussion of Kottwitz's formula, and now solve the arithmetic volume problem mentioned in the introduction.
For simplicity we will write $K$ rather than $K_f$ for open compact subgroups of $G(\Aff_f)$ in this section.

\begin{defn}
For $K$ a compact open subgroup of $G(\Aff_f)$, we define
\begin{equation*}
\chi_{K}(G)=\ol{v}(G; dg_{\infty})^{-1} \vol_{dg_f}(K)^{-1} \tau(G) d(G),
\end{equation*}
if $G$ is cuspidal.  If $G$ is not cuspidal, then $\chi_K(G)=0$.
\end{defn}
 
Note that if $K_0$ is another compact open subgroup of $G(\Aff_f)$, with $K \subseteq K_0$ of finite index, then $\chi_K(G)=[K_0:K] \chi_{K_0}(G)$.   In this section we compute the quantities $\chi_{K}(G)$ under some mild hypotheses on $G$.   

\subsection{Statement of Theorem}

Before getting embroiled in details, let us sketch the idea of the computation of $\chi_K(G)$.  The computation is considerably easier if $K$ is sufficiently small.  In this case, $\chi_K(G)$ is the classical Euler characteristic of a Shimura variety.  This in turn may be written in terms of Euler characteristics of an arithmetic subgroup of $G_{\ad}(\R)$.    For $G$ a semisimple and simply connected Chevalley group, such Euler characteristics were computed in \cite{Harder}.

Our work is to reduce to this case.  Given a compact open subgroup $K_0$ of $G(\Aff_f)$, we will pick a sufficiently small subgroup $K$ of $K_0$.  By the above we know the analogue of $\chi_K(G)$ for $G^{\simp}$.   To compute $\chi_{K_0}(G)$ we have two tasks:  to transition between $G$ and $G^{\simp}$, and to transition between $K$ and $K_0$.  
 
The resulting formula entails several standard definitions:

\begin{defn} Write $G(\R)_+ \subseteq G(\R)$ for the inverse image of $G_{\ad}(\R)^+$. Let $G(\Q)_+=G(\Q) \cap G(\R)_+$.  Write $\nu: G \twoheadrightarrow C$ for the quotient of $G$ by $G_{\der}$.  Let $C(\R)^\dagger=\nu(Z(\R))$, and $C(\Q)^\dagger=C(\Q) \cap C(\R)^\dagger$.
Write $\rho: G_{\simp} \to G_{\der}$ for the usual covering of $G_{\der}$ by $G_{\simp}$.
For $K$ a compact open subgroup of $G(\Aff_f)$, let $K^{\der}=G_{\der}(\Aff_f) \cap K$, and let $K^{\simp}$ be the preimage of $K$ in $G_{\simp}(\Aff_f)$.  Let $\Gamma_K=G(\Q)_+ \cap K$, let
$\Gamma^{\der}_K=G_{\der}(\Q)_+ \cap K$, let $\Gamma^{\simp}_K=K^{\simp} \cap G_{\simp}(\Q)_+$, and write $\Gamma^{\ad}_K$ for the image of $\Gamma_K$ in $G_{\ad}(\Q)$.
\end{defn}

In this section we avoid certain awkward tori for simplicity, preferring the following kind: 

\begin{defn} A torus $T$ over $\Q$ is $\Q \R$-equitropic if the largest $\Q$-anisotropic torus in $T$ is $\R$-anisotropic.
\end{defn}
Here are some basic facts about $\Q \R$-equitropic tori.
\begin{prop}
If $T$ is a $\Q \R$-equitropic torus then $T(\Q)$ is discrete in $T(\Aff_f)$.
If $G$ is a reductive group, and the connected component $Z^\circ$ of the center of $G$ is $\Q \R$-equitropic, then its derived quotient $C$ is also $\Q \R$-equitropic.
\end{prop}

\begin{proof} The first statement follows from Theorem 5.26 of \cite{Milne}.  The second is straightforward.
\end{proof}

In \cite{Serre}, Serre introduces an Euler characteristic $\chi_{\alg}(\Gamma) \in \Q$ applicable to any group $\Gamma$ with a finite index subgroup $\Gamma_0$ which  is torsion-free and has finite cohomological dimension.  In particular, it applies to our congruence subgroups $\Gamma=\Gamma_K$. 
Here are some simple properties of $\chi_{\alg}$:
\begin{itemize}
\item For an exact sequence of the form
\begin{equation*}
1 \to A \to B \to C \to 1,
\end{equation*}
with $A,B$ and $C$ groups as above,
we have $\chi_{\alg}(B)=\chi_{\alg}(A) \cdot \chi_{\alg}(C)$. 

\item If $\Gamma$ is a finite group, then
$\chi_{\alg}(\Gamma)= |\Gamma|^{-1}$.
\end{itemize}

The theorem of this section relates $\chi_K(G)$ to $\chi_{\alg}(\Gamma^{\simp}_{K})$.  More precisely :

\begin{thm} \label{chi-theorem}
Let $G$ be a reductive group over $\Q$.  Assume that $G_{\simp}$ has no compact factors, and that the connected component $Z^{\circ}$ of the center of $G$ is $\Q \R$-equitropic.  Let $K_0 \subset G(\Aff_f)$ be a compact open subgroup.
Then $\chi_{K_0}(G)$ is equal to
\begin{equation*}
\frac{ |\ker (\rho(\Q))| [G_{\der}(\Aff_f):G_{\der}(\Q)_+ K^{\der}_0] \left[\Gamma^{\der}_{K_0}:G_{\der}(\Q)_+ \cap \rho(K_0^{\simp}) \right] \cdot [C(\Aff_f) : C(\Q)^\dagger  \nu(K_0)]}{ [G(\R): G(\R)_+] |\nu(K_0) \cap C(\Q)^\dagger|}          \chi_{\alg}(\Gamma^{\simp}_{K_0}).
\end{equation*}

\end{thm}

Here $\rho(\Q)$ denotes the map  $\rho(\Q): G_{\simp}(\Q) \to G(\Q)$ on $\Q$-points.  
The assumption on the absence of compact factors is needed for strong approximation, and is discussed in \cite{Milne}.

When $G_{\simp}$ is a Chevalley group, and $\Gamma^{\simp}_{K_0}=G_{\simp}(\Z)$, this reduces the problem to the calculation of Harder \cite{Harder}:

\begin{prop} \label{H2}
Let $G$ be a simply connected, semisimple Chevalley group over $\Z$.  Write $m_1, \ldots, m_r$ for the exponents of its Weyl group $\Omega$, and put $\Gamma=G(\Z)$.  We have
\begin{equation*}
\chi_{\alg}(\Gamma)=\left(-\half \right)^r|\Omega_{\R}|^{-1} \prod_{i=1}^r B_{m_i+1}.
\end{equation*}

\end{prop}

Here $B_n$ denotes the $n$th Bernoulli number.  Recall that $\Omega_{\R}$ is the real Weyl group of $G$.

\subsection{Shimura Varieties}

To prove Theorem \ref{chi-theorem}, we will use some basic Shimura variety theory, which may be found in \cite{Deligne} or \cite{Milne}.  Much of the theory holds only for $K$ sufficiently small.  For simplicity, we will say ``$K$ is small'' rather than ``$K$ is a sufficiently small finite index subgroup of $K_0$''.

For convenience, we gather here many simplifying properties of small $K$, which we will often use without comment.  For the rest of this section assume that  $Z(G)^\circ$ is $\Q \R$-equitropic, 
and that $G_{\simp}$ has no compact factors. 

\begin{prop}  \label{small} Let $K$ be small. Then
\begin{enumerate}  
\item \label{lily} $K \cap Z(\Q)= \{1\}$.
\item $\nu(K) \cap C(\Q)= \{1\}$.
\item \label{abby} $G(\Q) \cap K G_{\der}(\Aff_f) \subseteq G_{\der}(\Q)$.
\item $G_{\der}(\Aff_f) \cap G(\Q) K= G_{\der}(\Q) K_{\der}$.
\item \label{joy} $K \cap G_{\der}(\Q) \subseteq \rho(G_{\simp}(\Q))$.
\item \label{reu} $K \cap G(\Q) \subseteq G(\Q)^+$.
\end{enumerate}
\end{prop}

\begin{proof}
The first two items follow because $Z^\circ$, and thus $C$, are $\Q \R$-equitropic.  Item \ref{abby} follows from Corollaire 2.0.12 in \cite{Deligne}, and the next item is a corollary.  Items \ref{joy} and \ref{reu} follow from Corollaire 2.0.5 and 2.0.14 in \cite{Deligne}, respectively.
\end{proof}

Recall that we have chosen a maximal compact subgroup $K_{\R}$ of $G(\R)$. 

\begin{defn}  Let
\begin{equation*}
X=G(\R)/ K_{\R}^+ Z(\R),
\end{equation*}

\begin{equation*}
\ol{X}=G(\R) / K_{\R} Z(\R),
\end{equation*}

and
\begin{equation*}
S_K=G(\Q) \bks X \times G(\Aff_f)/K
\end{equation*}
be the double coset space obtained through the action $q(x,g)k=(qx,qgk)$ of $q \in G(\Q)$ and $k \in K$.

Similarly, let
\begin{equation*}
\ol{S}_K=G(\Q) \bks \ol{X} \times G(\Aff_f)/K,
\end{equation*}
with the action of $G(\Q) \times K$ defined in the same way.

The component group of $S_K$ is finite and given (see 2.1.3 in \cite{Deligne}) by
\begin{equation} \label{pease}
\pi_0(S_K)= G(\Aff_f)/ G(\Q)_+ K.
\end{equation}
\end{defn}

There is some variation in the literature regarding the use of $X$ versus $\ol{X}$.  Deligne and Milne implicitly use $X$ in \cite{Deligne} and \cite{Milne}   (in light of Proposition 1.2.7 in \cite{Deligne}).  Harder uses $\ol{X}$ in \cite{Harder}.  Arthur uses 
\begin{equation*}
G(\R)/K_{\R}'
\end{equation*} 
in \cite{A-L^2}.   (Recall that $K_{\R}'=A_G(\R)^{+}K_{\R}$.) Since for us $Z^{\circ}$ is $\Q \R$-equitropic, we have 
\begin{equation*}
K_{\R}'=Z(\R)K_{\R},
\end{equation*}
and so this quotient is equal to $\ol{X}$.  

Since we would like to combine results stated in terms of $X$ with others stated in terms of $\ol{X}$, we must understand the precise relationship between the two.  This is the purpose of Proposition \ref{fibre computation} below.

\begin{defn} Let $G$ be a real group, and $Z$ its center.  Write
\begin{equation}
\ad: G(\R) \to G(\R)/Z(\R)
\end{equation}
for the quotient map.
\end{defn}
Note that $\ad(G(\R))$ has finite index in $G_{\ad}(\R)$.

\begin{lemma} \label{real groups}  For this lemma, let $G$ be a Zariski-connected reductive real group, and $K_{\R}$  a maximal compact subgroup of $G(\R)$.  Let $L_{\R}$ be a maximal compact subgroup of $G_{\ad}(\R)$ containing $\ad(K_{\R})$.  Then the following hold:
\begin{enumerate}
\item \label{do} $K_{\R}$ meets all the connected components of $G(\R)$.
\item \label{re} $K_{\R} \cap G(\R)^+=K_{\R}^+$.
\item \label{mi} $\ad(K_{\R})$ is a maximal compact subgroup of $\ad(G(\R))$.
\item \label{fa}   $\ad(K_{\R}^+)=L_{\R}^+$.
\item \label{so} $K_{\R}Z(\R) \cap G(\R)_+=K_{\R}^+Z(\R)$.
\end{enumerate}
\end{lemma}

\begin{proof} The first two statements follow from the Cartan decomposition (Corollary 4.5  in  \cite{Satake}).  

For the third statement, suppose that $C$ is a subgroup of $G(\R)$ with $\ad(K_{\R}) \subseteq \ad(C)$ and $\ad(C)$ compact.  If $\ad(K_{\R}) \neq \ad(C)$, then there is an element $a \in CZ(\R)-K_{\R}Z(\R)$.  By the Cartan decomposition, we may assume that $a=\exp(H)$, with $H$ a semisimple element of $\Lie(G)$, and $\alpha(H)$ real and nonnegative for every root $\alpha$ of $G$.  Since $a \notin Z(\R)$, we have $\alpha(H)>0$ for some root $\alpha$.  Thus $\ad(C)$ is not compact, a contradiction.  Thus $\ad(K_{\R})=\ad(C)$, and statement \ref{mi} follows.

For the fourth statement, note that $L_{\R} \cap \ad(G)=\ad(K_{\R})$, and so $L_{\R}/\ad(K_{\R})$ injects into $G_{\ad}(\R)/\ad(G(\R))$.  It follows that $\ad(K_{\R}^+)$ has finite index in $L_{\R}$.  Since it is connected, statement \ref{fa} follows.

For the fifth statement, let $g \in K_{\R}Z(\R) \cap G(\R)_+$.  Then $\ad(g) \in L_{\R} \cap   G_{\ad}(\R)^+$, so by statement \ref{re}, we see $\ad(g) \in L_{\R}^+=\ad(K_{\R}^+)$.  Thus $g \in K_{\R}^+Z(\R)$.   The other inclusion is obvious.
\end{proof}

\begin{prop} \label{fibre computation}
\begin{enumerate}  
\item The natural projection $p_X: X \to \ol{X}$ has fibres of order $[G(\R): G(\R)_+]$.
\item Let $X^+$ be a connected component of $X$.  It is stabilized by $G(\R)_+$, and the restriction of $p_X$ to $X^+$ is a $G(\R)_+$-isomorphism onto $\ol{X}$.
\item Let $K$ be small.  Then the natural projection $p_S: S_K \to \ol{S}_K$ has fibres of order $[G(\R): G(\R)_+]$.

\end{enumerate}
\end{prop}

\begin{proof}  Consider the natural map
\begin{equation} \label{amap}
 K_{\R}Z(\R)/K_{\R}^+Z(\R) \to G(\R)/G(\R)_+.
 \end{equation}
 It is surjective because $K_{\R}$ meets every connected component of $G(\R)$.  It is injective because $K_{\R}Z(\R) \cap G(\R)_+ \subseteq K_{\R}^+Z(\R)$.  
It follows that (\ref{amap}) is an isomorphism, and the first statement follows.
 
We now prove the second statement.  Note that $p_X$ is both an open and closed map so that $p_X(X^+)$ is a component of $\ol{X}$.  Since $K_{\R}$ meets every connected component of $G(\R)$, the set $\ol{X}$ is connected.  Therefore $p_X(X^+)=\ol{X}$.  By Proposition 5.7 of \cite{Milne}, there are $[G(\R): G(\R)_+]$ connected components of $X$, each stabilized by $G(\R)_+$.  Thus the fibre over a point in $\ol{X}$ is comprised of exactly one point from each component of $X$.  So $p_X$ restricted to $X^+$ is an isomorphism; it is clear that it respects the $G(\R)_+$-action.
 
To prove the third statement, we require $K$ to be sufficiently small, in the following way.  Suppose $K_*$ is an open compact subgroup of $G(\Aff_f)$ satisfying $K_* \cap G(\Q) \subseteq G(\Q)^+$.
Let $g_1, \ldots, g_r$ be representatives of the finite quotient group $G(\Q) K_*\bks G(\Aff_f)$.  Then we require that 
\begin{equation} \label{amina}
K \subseteq \bigcap_{i=1}^r   g_i^{-1}K_* g_i.
\end{equation}

Now for $x \in X$, let $\Fib(x)$ be the fibre of $p_X$ containing $x$.  If we further fix $g \in G(\Aff_f)$, let $\Fib(x,g)$ be the fibre of $p_S$ containing $(x,g)$.  (Here we understand $(x,g)$ as an element of $S_K$.)  We claim that for all such $x$ and $g$, the map
\begin{equation} \label{fibremap}
\Fib(x) \to \Fib(x,g)
\end{equation}
given by $x' \mapsto (x',g)$ is a bijection.  This will imply the third statement.

For surjectivity of (\ref{fibremap}), pick $(x',g') \in \Fib(x,g)$.  Then there are $q \in G(\Q)$ and $k \in G(\Aff_f)$ so that $qp_X(x')=p_X(x)$ and $qg'k=g$.  Let $x''=qx'$.  Then $x'' \in \Fib(x)$ and $(x'',g)=(x',g')$.

For injectivity of (\ref{fibremap}), suppose that $(x_1,g)=(x_2,g)$ in $S_K$ with $x_1,x_2 \in \Fib(x)$. Then in particular,
there is an element $q \in G(\Q)$ and $k \in K$ so that $qgk=g$ and $qx_1=x_2$.  Write $g=q_0k_0g_i$ with $q_0 \in G(\Q)$ and $k_0 \in K_*$.  Then we have
\begin{equation*}
q(q_0k_0g_i)k=q_0k_0g_i,
\end{equation*}
which we rewrite as 
\begin{equation*}
q_0^{-1}q q_0= k_0 g_i k^{-1}g_i^{-1}k_0^{-1}.
\end{equation*}
Using this and (\ref{amina}) we see that $q_0^{-1}qq_0 \in G(\Q) \cap K_* \subseteq G(\Q)^+$.  Since $G(\Q)^+$ is normal in $G(\Q)$, in fact $q \in G(\Q)^+$.
 
Meanwhile, pick $\xi_1,\xi_2 \in G(\R)$ representing $x_1$ and $x_2$, respectively.  Since $x_1,x_2 \in \Fib(x)$ we have $\xi_1^{-1}\xi_2 \in K_{\R}Z(\R)$.  Write $\xi_2=\xi_1 kz$, with $k \in K_{\R}$ and $z \in Z(\R)$.  Since $qx_1=x_2$, we have
$\xi_2^{-1}q \xi_1 \in K_{\R}^+Z(\R)$, and therefore $z^{-1}k^{-1}\xi_1^{-1}q\xi_1 \in K^+_{\R}Z(\R)$.  Using the fact that $q$ is in the normal subgroup $G(\R)_+$ of $G(\R)$, it follows that $k \in G(\R)_+ \cap K_{\R} \subseteq K_{\R}^+ Z(\R)$.  Thus  $x_1=x_2$, as desired. 

 \end{proof}

We will need the following two known results:   

\begin{prop} (Harder, see \cite{Harder}, \cite{Serre}) \label{H1} If $G$ is semisimple and $K$ is small, then $\chi_{\Top}( \Gamma_K \bks \ol{X})=\chi_{\alg}(\Gamma_K)$.
\end{prop}

\begin{prop}  \label{???} (See \cite{A-L^2}, \cite{GKM}) If $K$ is small, then $\chi_K(G)=\chi_{\Top}(\ol{S}_K)$. \end{prop}

\subsection{Computations}

The next three lemmas will allow us to convert our computation for $K_0$ to a computation for $K$.

\begin{lemma} \label{frank} If $K$ is small, then
\begin{equation*}
| C(\Q)^\dagger \bks C(\Aff_f) / \nu(K)|=[\nu(K_0): \nu(K)] |\nu(K_0) \cap C(\Q)^\dagger|^{-1} | C(\Q)^\dagger \bks C(\Aff_f) / \nu(K_0)|.
\end{equation*}
\end{lemma}

\begin{proof}
This follows from the exactness of the sequence
\begin{equation*}
1 \to \nu(K_0) \cap C(\Q)^\dagger \to \frac{\nu(K_0)}{\nu(K)} \to  C(\Q)^\dagger \bks C(\Aff_f) / \nu(K) \to  C(\Q)^\dagger \bks C(\Aff_f) / \nu(K_0) \to 1.
\end{equation*}
\end{proof}

\begin{lemma} \label{jeremiah} If $K \subseteq K_0$ is small, then
\begin{equation} \label{jade}
[\Gamma^{\ad}_{K_0}: \Gamma^{\ad}_K]=\frac{ [\Gamma_{K_0}: \rho(\Gamma_{K_0}^{\simp})][K_0: K]}{|K_0  \cap Z(\Q)|  [\nu(K_0) : \nu(K)] [K_0^{\der}: K^{\der} \rho(K_0^{\simp})] }.
\end{equation}
\end{lemma}

In the proof we refer to conditions of Proposition \ref{small}.

\begin{proof}
Consider the map
\begin{equation*}
\Gamma^{\der}_{K_0}/ \Gamma^{\der}_K \to \Gamma^{\ad}_{K_0} / \Gamma^{\ad}_K.
\end{equation*}

The kernel of this map sits in the middle of the exact sequence
\begin{equation*}
1 \to  \Gamma^{\der}_{K_0} \cap Z(\Q) \to (\Gamma_K Z(\Q) \cap \Gamma^{\der}_{K_0})/\Gamma^{\der}_K \to (\Gamma_K Z(\Q) \cap \Gamma^{\der}_{K_0}) \left/ \Gamma^{\der}_K (\Gamma^{\der}_{K_0} \cap Z(\Q)) \right. \to 1,
\end{equation*}
using condition \ref{lily}.
This last quotient is trivial, because actually $\Gamma_K=\Gamma_K^{\der}$ by condition \ref{abby}.

We have established the exactness of the sequence
\begin{equation*}
1 \to \Gamma^{\der}_{K_0} \cap Z(\Q) \to \Gamma^{\der}_{K_0} / \Gamma^{\der}_K \to \Gamma^{\ad}_{K_0} / \Gamma^{\ad}_K \to
\Gamma_{K_0}Z(\Q)/ \Gamma^{\der}_{K_0}Z(\Q) \to 1.
\end{equation*}
The last quotient is isomorphic to
\begin{equation*}
\Gamma_{K_0} \left / (Z(\Q) \cap K_0)\Gamma^{\der}_{K_0} \right.,
\end{equation*}
which itself sits inside the exact sequence
\begin{equation*}
1 \to K_0 \cap Z(\Q) \left / \Gamma^{\der}_{K_0} \cap Z(\Q) \right. \to \Gamma_{K_0} / \Gamma^{\der}_{K_0} \to
\Gamma_{K_0} \left / (Z(\Q) \cap K_0)\Gamma^{\der}_{K_0} \right. \to 1.
\end{equation*}

The quantity $|\Gamma^{\der}_{K_0} \cap Z(\Q)|$ cancels, and it follows that
\begin{equation} \label{ac/dc}
 [\Gamma^{\ad}_{K_0}: \Gamma^{\ad}_K]=\frac{  [\Gamma^{\der}_{K_0}: \Gamma^{\der}_K] \cdot [\Gamma_{K_0}:\Gamma^{\der}_{K_0}]}{|K_0 \cap Z(\Q)|}.
\end{equation}

By condition \ref{joy} we have
\begin{equation*}
 1 \to \rho(\Gamma_{K_0}^{\simp})/ \rho(\Gamma_K^{\simp}) \to \Gamma_{K_0}^{\der}/ \Gamma_K^{\der} \to \Gamma_{K_0}^{\der}/ \rho(\Gamma_{K_0}^{\simp}) \to 1.
\end{equation*}

Strong approximation tells us that $G_{\simp}(\Q)$ is dense in $G_{\simp}(\Aff_f)$.  Therefore we have isomorphisms
\begin{equation*}
\rho(\Gamma_{K_0}^{\simp})/ \rho(\Gamma_K^{\simp}) \isom   \Gamma^{\simp}_{K_0}/ \Gamma^{\simp}_K \isom K_0^{\simp}/ K^{\simp} \isom \rho(K_0^{\simp})/ \rho(K^{\simp}).
\end{equation*}
Combining this with the exact sequences
\begin{equation*}
1 \to K_0^{\der}/ K^{\der} \to K_0/K \to \nu(K_0)/\nu(K) \to 1
\end{equation*}
and
\begin{equation} \label{custer}
1 \to \rho(K_0^{\simp})/ \rho(K^{\simp}) \to K_0^{\der}/ K^{\der} \to K_0^{\der} \left / K^{\der} \rho(K_0^{\simp}) \right. \to 1,
\end{equation}
we obtain
\begin{equation*}
[\Gamma_{K_0}^{\der}: \Gamma_K^{\der}]=\frac{[\Gamma_{K_0}^{\der}: \rho(\Gamma_{K_0}^{\simp})][K_0:K]}
{[K_0^{\der}: K^{\der} \rho(K_0^{\simp})][\nu(K_0): \nu(K)]}.
\end{equation*}
Plugging this into (\ref{ac/dc}) gives the lemma.
\end{proof}

\begin{cor} \label{reparation}
Suppose that $K \subseteq K_0$ is small, and $g \in G(\Aff_f)$ with $gKg^{-1} \subseteq K_0$ also small.  Then
\begin{equation*}
[\Gamma^{\ad}_{K_0}: \Gamma^{\ad}_{gKg^{-1}}]=[\Gamma^{\ad}_{K_0}: \Gamma^{\ad}_K].
\end{equation*}
\end{cor}

\begin{proof}
We show that the expression (\ref{jade}) does not change when $K$ is replaced with $gKg^{-1}$. 
Clearly $\nu(K)=\nu(gKg^{-1})$.  Since
\begin{equation*}
[K_0:K]= \frac{\vol_{dg_f}(K_0)}{\vol_{dg_f}(K)},
\end{equation*}
we have $[K_0:gKg^{-1}]=[K_0:K]$.
Finally, we claim that
\begin{equation*}
 [K_0^{\der}: (gKg^{-1})^{\der} \rho(K_0^{\simp})] = [K_0^{\der}: K^{\der} \rho(K_0^{\simp})] .
\end{equation*}
From the exact sequence (\ref{custer}), it is enough to show that $[K_0^{\der}:(gKg^{-1})^{\der}]=[K_0^{\der}:K^{\der}]$ and $[\rho(K_0^{\simp}):\rho((gKg^{-1})^{\simp})]=[\rho(K_0^{\simp}):\rho(K^{\simp})]$.  These hold because $(gKg^{-1})^{\der}=gK^{\der}g^{-1}$ and $\rho((gKg^{-1})^{\simp})=g\rho(K^{\simp})g^{-1}$.
\end{proof} 

\begin{lemma} \label{schloendorn}
 If $G$ is semisimple and $K$ is small, then
\begin{equation*}
|\pi_0(S_K)|=  [K_0: K \rho(K_0^{\simp})][ \Gamma_{K_0}:  G(\Q)_+ \cap \rho(K_0^{\simp})]  |\pi_0(S_{K_0})|.
\end{equation*}
\end{lemma}
\begin{proof}

The kernel of the projection $\pi_0(S_K) \twoheadrightarrow \pi_0(S_{K_0})$ is isomorphic to
\begin{equation*}
K_0 \left / (KG(\Q)_+ \cap K_0) \right. .
\end{equation*}
By Section 2.1.3 in \cite{Deligne}, we have $ \rho(G_{\simp}(\Aff_f)) \subseteq K G(\Q)_+$.   
Using the exact sequence
\begin{equation*}
1 \to (K_0 \cap KG(\Q)_+) \left /K \rho(K_0^{\simp}) \right. \to K_0 \left / K \rho(K_0^{\simp}) \right. \to K_0 \left/ (KG(\Q)_+ \cap K_0) \right. \to 1,
\end{equation*}
we are reduced to computing the order of
\begin{equation*}
(K_0 \cap KG(\Q)_+) \left / K \rho(K_0^{\simp}) \right. \isom  \Gamma_{K_0} \left/ (K \rho(K_0^{\simp}) \cap G(\Q)_+) \right..
\end{equation*}
This group sits in the sequence
\begin{equation*}
1 \to (G(\Q)_+ \cap K \rho(K_0^{\simp})) \left / (G(\Q)_+ \cap \rho(K_0^{\simp})) \right. \to \Gamma_{K_0} \left / (G(\Q)_+ \cap \rho(K_0^{\simp})) \right. \to
\Gamma_{K_0} \left / (K \rho(K_0^{\simp}) \cap G(\Q)_+) \right. \to 1.
\end{equation*}

We claim the kernel is trivial.  Note that $K \rho(K_0^{\simp}) \subseteq K \rho(G_{\simp}(\Q) K^{\simp})$ by strong approximation.  So
\begin{equation*}
\begin{split}
G(\Q)_+ \cap K \rho(K_0^{\simp}) &\subseteq G(\Q)_+ \cap K \rho(G_{\simp}(\Q)) \\
                                &= G(\Q)_+ \cap (K \cap G(\Q)) \rho(G_{\simp}(\Q)).
\end{split}
\end{equation*}
Since $K \cap G(\Q) \subseteq \rho(G_{\simp}(\Q))$ by condition \ref{joy} of Proposition \ref{small}, we have $G(\Q)_+ \cap K \rho(K_0^{\simp}) \subseteq G(\Q)_+ \cap \rho(K_0^{\simp})$.  This proves the claim, and the lemma follows.

\end{proof}

In the course of proving the theorem, we will pass to the adjoint group to apply Harder's theorem (Proposition \ref{H1}), but lift to $G_{\simp}$ to apply Harder's calculation (Proposition \ref{H2}).    We must record the difference between Serre's Euler characteristic at $G_{\ad}$ and $G_{\simp}$.

\begin{lemma} \label{kristin}
We have
\begin{equation*}
\chi_{\alg}(\Gamma_{K_0}^{\ad})= \frac{|\ker (\rho(\Q))| | K_0 \cap Z(\Q)|}{[\Gamma_{K_0}^{\der}: \rho(\Gamma_{K_0}^{\simp})][\Gamma_{K_0}: \Gamma_{K_0}^{\der}]} \chi_{\alg}(\Gamma_{K_0}^{\simp}).
\end{equation*}
\end{lemma}

\begin{proof} This follows from the properties of $\chi_{\alg}$ mentioned earlier.
\end{proof}

\subsection{Proof of Theorem \ref{chi-theorem}}

\begin{proof}
Pick a set $g_1, \ldots, g_r$ of representatives of $\pi_0(S_{K_0})$, viewed as a quotient of $G(\Aff_f)$ as in (\ref{pease}).

Let $K$ be small subgroup of finite index in $K_0$.   Possibly by intersecting finitely many conjugates of $K$, we may assume
\begin{itemize}
\item $K$ is normal in $K_0$
\item $g_iKg_i^{-1}$ is a small subgroup of $K_0$ for all $i$.
\end{itemize}

By Proposition \ref{???}, $\chi_K(G)=\chi_{\Top}(\ol{S}_K)$. By Proposition \ref{fibre computation}, this is equal to $[G(\R): G(\R)_+]^{-1}\chi_{\Top}(S_K)$.  Write $\Gamma_g$ for $\Gamma^{\ad}_{gKg^{-1}}$. By 2.1.2 of \cite{Deligne}, the components of $S_K$ are each isomorphic to $\Gamma_g \bks X^+$, where $X^+$ is a component of $X$.   Here $g$ runs over $\pi_0(S_K)$.    

By Proposition \ref{fibre computation}, the topological spaces $\Gamma_g \bks X^+$ and $\Gamma_g \bks \ol{X}$ are isomorphic.  Therefore we have
\begin{equation*}
\chi_{\Top}(\Gamma_g \bks X^+) = \chi_{\Top}(\Gamma_g\bks \ol{X}).
\end{equation*}

Applying Proposition \ref{H1} to $G_{\ad}$, this is equal to $\chi_{\alg}(\Gamma_g) $.  Therefore
\begin{equation*}
\chi_K(G)=  [G(\R): G(\R)_+]^{-1} \sum_{g \in \pi_0(S_K)} \chi_{\alg}(\Gamma_g).
\end{equation*}

Every element in $\pi_0(S_K)$ may be written as the product of an element of $\pi_0(S_{K_0})$ with an element of $K_0$.
Since $K$ is normal in $K_0$, the groups $\Gamma_{gk_0}$ and $\Gamma_{g}$ are equal for $k_0 \in K_0$.  It follows that
\begin{equation*}
\chi_K(G)=\frac{|\pi_0(S_{K})|}{ [G(\R): G(\R)_+]|\pi_0(S_{K_0})|}        \sum_{i=1}^r \chi_{\alg}(\Gamma_{g_i}).
\end{equation*}

By Corollary \ref{reparation} we have
\begin{equation*}
\begin{split}
\chi_{\alg}(\Gamma_{g_i}) &=[\Gamma^{\ad}_{K_0}: \Gamma^{\ad}_{g_iKg_i^{-1}}] \chi_{\alg}(\Gamma^{\ad}_{K_0}) \\
 &=[\Gamma^{\ad}_{K_0}: \Gamma^{\ad}_K] \chi_{\alg}(\Gamma^{\ad}_{K_0}). 
\end{split}
\end{equation*}

This gives
\begin{equation*}
\chi_K(G)=  [G(\R): G(\R)_+]^{-1}  [\Gamma^{\ad}_{K_0}: \Gamma_K^{\ad}]  \left| \pi_0(S_K) \right|           \chi_{\alg}(\Gamma^{\ad}_{K_0}).
 \end{equation*}

The component group $\pi_0(S_K)$ fits into the exact sequence
\begin{equation*}
1 \to G_{\der}(\Aff_f) / (G_{\der}(\Aff_f) \cap G(\Q)_+K) \to \pi_0(S_K) \to C(\Q)^\dagger \bks C(\Aff_f) / \nu(K) \to 1
\end{equation*}

This gives
\begin{equation*}
\chi_K(G)=[G(\R): G(\R)_+]^{-1}  |\pi_0(S_{K^{\der}})|  |C(\Q)^\dagger \bks C(\Aff_f) / \nu(K)| [\Gamma^{\ad}_{K_0}: \Gamma_K^{\ad}]  \chi_{\alg}(\Gamma^{\ad}_{K_0}).
\end{equation*}
where here $\pi_0(S_{K^{\der}})=G_{\der}(\Aff_f)/ G_{\der}(\Q)_+ K^{\der}$.

Using  $\chi_{K_0}(G)=[K_0:K]^{-1} \chi_K(G)$ together with Lemma \ref{frank} gives
\begin{equation*}
\chi_{K_0}(G)= \frac{|\pi_0(S_{K^{\der}})|  [\nu(K_0):\nu(K)]  |C(\Q)^\dagger \bks C(\Aff_f) / \nu(K_0)|
[\Gamma^{\ad}_{K_0}: \Gamma_K^{\ad}]}
{[G(\R): G(\R)_+]  |\nu(K_0) \cap C(\Q)^\dagger|[K_0:K]} \chi_{\alg}(\Gamma^{\ad}_{K_0}).
\end{equation*}

By Lemmas \ref{jeremiah} and \ref{kristin},  
 
\begin{equation*}
\chi_{K_0}(G)=\frac{ |\ker (\rho(\Q))| |\pi_0(S_{K^{\der}})|  |C(\Q)^\dagger \bks C(\Aff_f) / \nu(K_0)|}{ [G(\R): G(\R)_+] |\nu(K_0) \cap C(\Q)^\dagger|[ K_0^{\der}: K^{\der}\rho(K_0^{\simp})]}          \chi_{\alg}(\Gamma^{\simp}_{K_0}).
\end{equation*}

The theorem then follows from Lemma \ref{schloendorn}.

\end{proof}

\subsection{Examples}

We now use Theorem \ref{chi-theorem} and Proposition \ref{H2} to explicitly compute some cases of $\chi_{K_0}(G)$.   Recall that we write $\OO_f$ for the integer points of $\Aff_f$.
 
\begin{cor} If $T$ is a torus and $K_0 \subset T(\Aff_f)$ is a compact open subgroup, then
\begin{equation*}
\chi_{K_0}(T)=|T(\Q) \bks T(\Aff_f) / K_0| \cdot |K_0 \cap T(\Q)|^{-1}.
\end{equation*}
\end{cor}

$\bullet$ Let $T=\Gm$, and $K_0=T(\OO_f)$.  Then $\chi_{K_0}(T)=\half$.

$\bullet$ Let $T$ be the norm-one subgroup of an imaginary quadratic extension $E$ of $\Q$.  Let $K_0=T(\OO_f)$.  Write $\OO(E)$ for the integer points of the adeles $\Aff_E$ over $E$.  Then $T(\Q) \bks T(\Aff_f) / K_0$ injects into $E^\times \bks \Aff_{E,f}^\times/ \OO(E)^{\times}$, which is in bijection with the class group.  If the class number of $E$ is trivial, it follows that $\chi_{K_0}(T)=|T(\Z)|^{-1}$.

\begin{cor} If $G$ is semisimple and simply connected, then
\begin{equation*}
\chi_{K_0}(G)= [G(\R): G(\R)_+]^{-1} \chi_{\alg}(\Gamma_{K_0}).
\end{equation*}
\end{cor}

$\bullet$ Let $G=\SL_2$ and $K_0=G(\OO_f)$. Then 
\begin{equation*}
\begin{split}
\chi_{K_0}(G) &=\chi_{\alg}(\SL_2(\Z)) \\
&= - \half B_2 \\
&=-2^{-2}3^{-1}.
\end{split}
\end{equation*}

$\bullet$ Let $G=\Sp_4$ and $K_0=G(\OO_f)$.  Then 
\begin{equation*}
\begin{split}
\chi_{K_0}(G) &=\chi_{\alg}(\Sp_4(\Z))\\
&= -\frac{1}{8}B_2 B_4\\
&=-2^{-5}3^{-2}5^{-1}. 
\end{split}
\end{equation*}

\bigskip
When the derived group is simply connected the calculation is not much harder.

\begin{cor} If $G_{\der}$ is simply connected, then
\begin{equation*}
\chi_{K_0}(G)=\frac{  |C(\Q)^\dagger \bks C(\Aff_f) / \nu(K_0)|}{ [G(\R): G(\R)_+] |\nu(K_0) \cap C(\Q)^\dagger|}          \chi_{\alg}(\Gamma^{\der}_{K_0}).
\end{equation*}
\end{cor}

$\bullet$ Let $G=\GL_2$ and $K_0=G(\OO_f)$.  Then $\chi_{K_0}(G)=\half \chi_{\alg}(\SL_2(\Z))=-2^{-3}3^{-1}$.

$\bullet$ Let $G=\GSp_4$ and $K_0=G(\OO_f)$.  Then $\chi_{K_0}(G)=\half \chi_{\alg}(\Sp_4(\Z))=-2^{-6}3^{-2}5^{-1}$

\begin{lemma} If all the points of $\ker \rho$ are $\Q$-rational, then $\left[\Gamma^{\der}_{K_0}:G_{\der}(\Q)_+ \cap \rho(K_0^{\simp}) \right]=1$.
\end{lemma}

\begin{proof}
By Section 2.0.3 of \cite{Deligne}, we have an injection
\begin{equation*}
G_{\der}(\Q) / \rho(G_{\simp}(\Q)) \hookrightarrow H^1(\im( \Gal(\ol{\Q}/\Q)), (\ker \rho)(\ol{\Q})),
\end{equation*}
using the cohomology group defined in that paper.  We also have an injection
\begin{equation*}
\Gamma^{\der}_{K_0} / (G_{\der}(\Q)_+ \cap \rho(K_0^{\simp})) \hookrightarrow G_{\der}(\Q) / \rho(G_{\simp}(\Q)).
\end{equation*}

Since all the points of $\ker \rho$ are $\Q$-rational, all these groups are trivial.
\end{proof}

$\bullet$ Let $G=\PGL_2$ and $K_0=G(\OO_f)$.  The only nontrivial factors in the formula are $[G(\R): G(\R)_+]=2$, $|\ker \rho(\Q)|=2$, and
$\chi_{\alg}(\SL_2(\Z))=-2^{-2}3^{-1}$.  Thus $\chi_{K_0}(G)= -2^{-2}3^{-1}$.

\section{The Case of $\SL_2$} \label{Special Linear}
 In this section we work out the example of $\SL_2$.
Let $G=\SL_2$, defined over $\Q$.  Let $A$ be the subgroup of diagonal matrices in $G$, and let $T$ be the maximal elliptic torus of $G$ given by matrices
\begin{equation} \label{rotation}
\gm_{a,b}=\left(\begin{array}{cc}
a   & -b \\
b  & a  \\
  \end{array} \right),
\end{equation}
with $a^2+b^2=1$.

The characters and cocharacters of $T$ are both isomorphic to $\Z$.  We identify $\Z \isom X^*(T)$ via $n \mapsto \chi_n$, where $\chi_n(\gm_{a,b})=(a+bi)^n$.  We specify  $\Z \isom X_*(T)$ by identifying $n$ with the cocharacter
taking $\alpha$ to $\left(\begin{array}{cc}
\alpha   & 0 \\
0  & \alpha^{-1}  \\
  \end{array} \right)$.
The roots of $T$ in $G$ are then $\{ \pm 2 \}$, and
the coroots of $T$ in $G$ are $\{ \pm 1 \}$.  The Weyl group $\Omega$ of these systems has order $2$ and the compact Weyl group $\Omega_\R$ is trivial.  Thus each $L$-packet of discrete series has order $2$.
The dual group to $G$ is $\hat{G}=\PGL_2(\C)$ in the usual way.

Pick an element $\xi \in G(\C)$ so that

\begin{equation*}
\Ad(\xi) \left( \begin{array}{cc}
a &-b \\
  b&a  \\  \end{array} \right)=\left( \begin{array}{cccc}
a+ib & \\
& a-ib \\  \end{array} \right),
\end{equation*}
 and put $B_T=\Ad(\xi^{-1})B_A$.  Then $B_T$ is a Borel subgroup of $G(\C)$ containing $T$.

Consider the Langlands parameter $\varphi_G: W_\R \to \hat{G}$ given by $\varphi_G(\tau)=\left(\begin{array}{cc}
0  & 1 \\
1  & 0  \\
  \end{array} \right) \times 1$, and
\begin{equation*}
\varphi_G(z)=\left(\begin{array}{cc}
z^n   & 0 \\
0  & \ol{z}^n  \\
  \end{array} \right) \times z =z^\mu \ol{z}^\nu \times z,
\end{equation*}
where $\mu$ corresponds to $n \in X_*(\hat{T}) \isom X^*(T)$ and $\nu$ corresponds to $-n$.  The corresponding representation $E$ of $G(\C)$ has highest weight $\lambda_B=n-1 \in X^*(T)$.  It is the $(n-1)$-th symmetric power representation.  Its central character is $\lambda_0(z)=z^{n-1}$, where $z=\pm 1$.

We put $\pi_G=\pi(\varphi_G,B_T)$, in the notation from Section \ref{Langlands packets}.  Write $\pi_G'$ for the other discrete series representation in $\Pi_E$.  Thus the $L$-packet determined by $\varphi_G$ is
\begin{equation*}
\Pi_E= \{ \pi_G, \pi_G' \}.
\end{equation*}

We will put $f_{\infty}dg_{\infty}=e_{\pi_G}$ as in Section \ref{PC}.

\subsection{Main Term}

First we consider the terms $ST_g(fdg,\pm 1)$.

We have $S \Phi_G(1,e_{\pi_G})=-n\ol{v}(G; dg_{\infty})^{-1}$, and so
\begin{equation*}
ST_g(fdg,\pm 1, G)=( \pm 1)^n n \ol{v}(G; dg_{\infty})^{-1} f^{\infty}(\pm 1).
\end{equation*}

We have $S \Phi_A(1,e_{\pi_G})=-\ol{v}(G; dg_{\infty})^{-1}$, and so
\begin{equation*}
ST_g(fdg,\pm 1 , A)=  (\pm 1)^n  \half \ol{v}(G; dg_{\infty})^{-1} f_A^\infty( \pm 1).
\end{equation*}

If $\gm$ is a regular semisimple element of $G(\C)$ with eigenvalues $\alpha, \alpha^{-1}$, then according to the Weyl character formula,
\begin{equation*}
\tr(\gm; E)= \frac{\alpha^n-\alpha^{-n}}{\alpha-\alpha^{-1}}.
\end{equation*}

Define
\begin{equation*}
t_4(n)=\tr \left( \left(\begin{array}{cc}
i   & 0 \\
0  & -i  \\
  \end{array} \right); E \right),
\end{equation*}
where $i$ is a fourth root of unity.
Then $t_4(n)=0$ if $n$ is even, and $t_4(n)=(-1)^{\frac{n-1}{2}}$ if $n$ is odd.

Similarly, define
\begin{equation*}
t_3(n)=\tr \left( \left(\begin{array}{cc}
\zeta   & 0 \\
0  & \zeta^2  \\
  \end{array} \right); E \right),
\end{equation*}
where $\zeta$ is a third root of unity.
Then $t_3(n)=[0,1,-1;3]_n$, meaning that  
\begin{equation*}
t_3(n) = 
\begin{cases} & 0 \text{ if } n \equiv 0, \\
                          & 1 \text{ if } n \equiv 1, \\
                          & -1 \text{ if } n \equiv 2. \end{cases}
                          \end{equation*}
Here the congruence is modulo $3$.

There are three stable conjugacy classes of elliptic $\gm \in G(\Q)$, which we represent by
\begin{equation*}
\gm_3=\left(\begin{array}{cc}
-1   & -1 \\
1  & 0  \\
  \end{array} \right),
\gm_4=\left(\begin{array}{cc}
0   & -1 \\
1  & 0  \\
  \end{array} \right),
 \text{ and }
\gm_6=\left(\begin{array}{cc}
0   & -1 \\
1  & 1  \\
  \end{array} \right).
\end{equation*}

Note that $-\gm_4 \sim \gm_4$, $\gm_6^2=\gm_3$, and $-\gm_3 \sim \gm_6$.

Write $T_3$ for the elliptic torus consisting of elements
\begin{equation*}
\left(\begin{array}{cc}
a   & a-b \\
b-a  & b  \\
  \end{array} \right),
\end{equation*}
with $a^2-ab+b^2=1$.

We have $S \Phi_G(\gm_3,e_{\pi_G})=-\ol{v}(T_3)^{-1} t_3(n)$, and so
\begin{equation*}
ST_g(fdg,\gm_3 , G)=-\ol{v}(T_3)^{-1} SO_{\gm_3}(f^\infty dg_f) t_3(n).
\end{equation*}

We have $S \Phi_G(\gm_4,e_{\pi_G})=-\ol{v}(T)^{-1} t_4(n)$, and so
\begin{equation*}
ST_g(fdg,\gm_4 , G)=-\ol{v}(T)^{-1} SO_{\gm_4}(f^\infty dg_f) t_4(n).
\end{equation*}

Finally $S \Phi_G(\gm_6,e_{\pi_G})=-\ol{v}(T_3) t_3(n)(- 1)^{n-1}$, and so
\begin{equation*}
ST_g(fdg,  \gm_6 , G)=-\ol{v}(T_3)^{-1} SO_{-\gm_3}(f^\infty dg_f) t_3(n)(- 1)^{n-1}.
\end{equation*}

Thus, $ST_g(fdg)$ is equal to the sum of the following terms:
\begin{equation*}
-n \ol{v}(G; dg_{\infty})^{-1} f^\infty(1)    +n \ol{v}(G; dg_{\infty})^{-1} f^\infty(-1) (- 1)^{n}
\end{equation*}
\begin{equation*}
-\half \ol{v}(A; da_{\infty})^{-1}f_A^\infty(1)+\half \ol{v}(A; da_{\infty})^{-1}f_A^\infty(-1)(-1)^{n}
\end{equation*}
\begin{equation*}
-\ol{v}(T_3)^{-1} SO_{\gm_3}(f^\infty dg_f) t_3(n)-\ol{v}(T)^{-1} SO_{\gm_4}(f^\infty dg_f) t_4(n)+\ol{v}(T_3)^{-1} SO_{-\gm_3}(f^\infty dg_f) t_3(n)(- 1)^{n}.
\end{equation*}

\subsection{Endoscopic Terms}

\begin{defn} Let $E$ be an imaginary quadratic extension of $\Q$.  Write $H_E$ for the kernel of the norm map $\Res^E_{\Q} \Gm \to \Gm$. \end{defn}

The $H_E$ comprise the (proper) elliptic endoscopic groups for $G=\SL_2$.  For each $H=H_E$ one finds $\tau(H)=2$ and $| \Out(H,s,\eta)|=1$ (see \cite{K84}, Section 7).    Therefore $\iota(G,H)=\half$.

\begin{equation*}
e_{\pi_G}^H=e_{\chi_n}+ e_{\chi_n^{-1}}.
\end{equation*}

Write $f^H dh=f^{\infty H} dh_f  e_{\pi_G}^H$, where $f^{\infty H}dh_f$ is the transfer of $f^{\infty} dg_f$.  Choose $dh_{\infty}$ so that $dh_f dh_{\infty}$ is the Tamagawa measure
on $H$.  Then we obtain
\begin{equation*}
ST_g(f^H dh)=2  \ol{v}(H; dh_{\infty}) \sum_{\gm_H} f^{\infty,H}(\gm_H)  \Tr_{\Q}^E(\gm_H^n), 
\end{equation*}
the sum being taken over $\gm_H \in H(\Q)$.

{ \bf Remark:}  Consider the local transfer, where $f_p dg_p$ is a spherical (invariant under $G(\Zp)$) measure on $G(\Qp)$.  Then if $H$ ramifies over $p$, a representation $\pi_p$ in one of the $L$-packets transferring from $H$ will also be ramified.  This means that $\tr \pi_p(f_pdg_p)=0$.  So we take $f_p^H=0$ in this case.
 Thus $\Kot(f dg)=ST_g(fdg)$; there is no (proper) endoscopic contribution.  This is compatible with the fact that $m_{\disc}$ is constant on $L$-packets in this case.

\subsection{Case of $\Gamma=\SL_2(\Z)$}

We take $K_f=K_0$ to be the integral points of $G(\Aff_f)$.  Also let $K_A=K_0 \cap A(\Aff_f)$ and $K_T=K_0 \cap T(\Aff_f)$.  Each of these breaks into a product of local groups $K_{0,p}$, etc.

We put $f^\infty dg_f=e_{K_0}$.  Note that $f^\infty(g)=f^\infty(-g)$ for all $g \in G(\Aff_f)$ and $f^\infty_A(a)=f^\infty_A(-a)$ for all $a \in A(\Aff_f)$.  Therefore, if $n$ is even, then $ST_g(fdg)=0$.  So assume henceforth that $n$ is odd.  Then our expression is equal to:
\begin{equation*}
-2n \ol{v}(G; dg_{\infty})^{-1} f^\infty(1)- \ol{v}(A; da_{\infty})^{-1}f_A^\infty(1) -2\ol{v}(T_3)^{-1} SO_{\gm_3}(f^\infty dg_f) t_3(n)+\ol{v}(T)^{-1} SO_{\gm_4}(f^\infty dg_f) (-1)^{\frac{n+1}{2}}.
\end{equation*}

We have
\begin{equation*}
\begin{split}
-2n \ol{v}(G; dg_{\infty})^{-1} f^\infty(1)  &= -2n \ol{v}(G; dg_{\infty})^{-1} \vol_{dg_f}(K_0)^{-1} \\
&= -2n \tau(G)^{-1} d(G)^{-1} \chi_{K_0}(G) \\
&=\dfrac{n}{12},
\end{split}
\end{equation*}
\begin{equation*}
\begin{split}
- \ol{v}(A; da_{\infty})^{-1} f_A^\infty(1) &= - \ol{v}(A; da_{\infty})^{-1} \vol_{da_f}(K_A)^{-1} \\
 &= -\tau(A)^{-1} d(A)^{-1} \chi_{K_A}(A) \\
&=-\half.
\end{split}
\end{equation*}

Now we consider $SO_{\gm_4}(f^\infty dg_f; dt_f)$.
We have $1-\alpha(\gm_4)=2$ for the positive root $\alpha$ of $G$.  Therefore by Proposition \ref{orbitalintegrals}, the local orbital integrals are equal to $\vol_{dt_p}(K_{T,2})^{-1}$ for $p \neq 2$.
At $p=2$, one has two stable conjugacy classes $\gm_4$ and $\gm_4'$ in the conjugacy class of $\gm_4$, where $\gm_4'=\left(\begin{array}{cc}
0   & 1 \\
-1  & 0  \\
  \end{array} \right)$.

It follows that

\begin{equation*}
SO_{\gm_4}(f^\infty dg_f; dt_f)= \left( O_{\gm_4}(e_{K_2}; dt_2)+ O_{\gm_4'}(e_{K_2}; dt_2) \right) \prod_{p \neq 2} \vol_{dt_p}(T(\Qp) \cap K_p)^{-1} .
\end{equation*}

To compute the local integral at $p=2$, we reduce to a $\GL_2$-computation by the following lemma.  Its proof is straightforward.

\begin{lemma}
Let $F$ be a $p$-adic local field with ring of integers $\OO$.  Put $G=\SL_2$, $\tilde{G}=\GL_2$,  and $Z$ for the center of $\tilde{G}$.  Pick Haar measures $dg$ on $G(F)$, $d \tilde g$ on $\tilde G(F)$, and $dz$ on $Z(F)$.
Let
$f \in C_c(Z(F) \bks \tilde G(F))$.  Then
\begin{equation*}
\frac{\vol_{dz}(Z(\OO))}{\vol_{d \tilde g}(\tilde G(\OO))} \int_{Z(F) \bks \tilde G(F)} f(g) \frac{d \tilde g}{dz}= \vol_{dg}(G(\OO))^{-1} |\OO^{\times}/\OO^{\times 2}|^{-1} \sum_{\alpha} \int_{G(F)} f(t_\alpha g) dg.
\end{equation*}
Here $\alpha$ runs over the square classes in $F^\times$, and $t_\alpha=\left(\begin{array}{cc}
\alpha   & 0 \\
0  & 1  \\
  \end{array} \right)$.
\end{lemma}

\begin{prop}
We have
\begin{equation*}
O_{\gm_4}(e_{K_2}; dt_2)+ O_{\gm_4'}(e_{K_2}; dt_2)= 2 \vol_{dt_2}(K_{T,2})^{-1}. 
\end{equation*}
\end{prop}

\begin{proof}
Write  $ \tilde f_2$ for the characteristic function of $\GL_2(\Z_2)Z(\Q_2)$.  By the lemma,
\begin{equation*}
\int_{Z(\Q_2) \bks \GL_2(\Q_2)}\tilde f_2(g^{-1}\gm_4 g) \frac{d \tilde g}{dz}=  \vol_{dt_2}(K_{T,2}) |\Z_2^{\times}/\Z_2^{\times 2}|^{-1} \sum_\alpha O_{\Ad(t_{\alpha})(\gm_4)}(e_{K_0}; dt_2).
\end{equation*}

Here we are normalizing $d \tilde g$ and $dz$ so that $\vol_{dz}(Z(\Z_2))=\vol_{d \tilde g}(\GL_2(\Z_2))=1$. 

In fact, $\Ad(t_{\alpha})(\gm_4)$ is conjugate in $G(\Q_2)$ to $\gm_4$ if and only if $\alpha$ is a norm from $\Q_2(\sqrt{-1})$, and in the contrary case, it is conjugate to $\gm_4'$.  It follows that

\begin{equation*}
\int_{Z(\Q_2) \bks \GL_2(\Q_2)} f_2(g^{-1}\gm_4 g) \frac{d \tilde g}{dz}=   \left(O_{\gm_4}(e_{K_2}; dt_2)+O_{\gm'_4}(e_{K_2}; dt_2) \right) \vol_{dt_2}(K_{T,2}).
\end{equation*}

By an elliptic orbital integral computation in \cite{KClay}, the left hand side is equal to $2$.

\end{proof}

We conclude that
\begin{equation*}
SO_{\gm_4}(f^\infty dg_f; dt_f)=2 \vol_{dt_f}(T(\Aff_f) \cap K_0)^{-1},
\end{equation*}
and so
\begin{equation*}
\begin{split}
-\ol{v}(T)^{-1} SO_{\gm_4}(f^\infty dg_f)t_4(n) &= -2\ol{v}(T)^{-1} \vol_{dt_f}(T(\Aff_f) \cap K_0)^{-1} t_4(n) \\
&= -2\tau(T)^{-1} \chi_{K_T}(T) t_4(n) \\
&=2^{-2} (-1)^{\frac{n+1}{2}}.
\end{split}
\end{equation*}

Similarly, we find that
\begin{equation*}
SO_{\gm_3}(f^\infty dg_f)=2 \vol_{dt_{3,f}}(T_3(\Aff_f) \cap K_0)^{-1},
\end{equation*}
and so
\begin{equation*}
-2\ol{v}(T_3)^{-1} SO_{\gm_3}(f^\infty dg_f) t_3(n)=-3^{-1}t_3(n).
\end{equation*}

We conclude that in this case,
\begin{equation*}
ST_g(fdg)= \frac{n}{12}-\half+ \frac{1}{4} (-1)^{\frac{n+1}{2}}-\frac{1}{3}t_3(n).
\end{equation*}

Note that for $n >1$ this agrees precisely with the discrete series multiplicities.  For $n=1$, this expression is equal to $-1$, but of course in this case $\pi$ is not regular.

\section{Real Tori} \label{Real Tori}

We have finished our discussion of $\SL_2$.  Starting with this section, we begin to work out the example of $\GSp_4$.  Various isomorphisms of tori must be written carefully, so we begin by explicitly working out their parametrizations.

\subsection{The Real Tori $\Gm$, $S$, and $T_1$}

We identify the group of characters of $\Gm$ with $\Z$ in the usual way, via $\left(a \mapsto a^n \right) \leftrightarrow n$.

Let $A_0= \Gm \times \Gm$, viewed as a maximal torus in $\GL_2$ in the usual way.  Via the above identification we obtain
  $X^*(A_0) \cong \Z^2$ and $X_*(A_0) \cong \Z^2$.

Let $S= \Res_\R^\C \Gm$.  Recall that $\Res^\C_\R \Gm$ denotes the algebraic group over $\R$ whose $\mathcal{A}$-points are $(\mathcal{A} \otimes \C)^\times$ for an $\R$-algebra $\mathcal{A}$.  By choosing the basis $\{1,i \}$ of $\C$ over $\R$, we have an injection $(\mathcal{A} \times \C)^\times \to \GL(\mathcal{A} \otimes \C) \cong \GL_2(\mathcal{A})$.  Thus we have an embedding $\iota_S:S \to \GL_2$ as an elliptic maximal torus.

There is a ring isomorphism $\varphi: \C \otimes \C \isom \C \times \C$ so that $\varphi(z_1 \otimes z_2)=(z_1z_2,z_1 \ol{z_2})$, which restricts to an isomorphism $\varphi: S(\C) \isom \Gm(\C) \times \Gm(\C)$.
This isomorphism is also actualized by conjugation within $\GL_2(\C)$.  Fix $x \in \GL_2(\C)$ so that
\begin{equation*}
\Ad(x) \left(\begin{array}{cc}
a   & -b \\
b  & a  \\
  \end{array} \right)=
\left(\begin{array}{cc}
a+ib   &  \\
  & a-ib  \\
  \end{array} \right);
\end{equation*}
then $\Ad(x): S(\C) \isom A_0(\C)$ is identical to $\varphi$, viewing these two tori under the embeddings above.

We fix the isomorphism from $\Z^2$ to $X^*(S)$ which sends $(1,0)$ (resp. $(0,1)$) to the character $\varphi$ composed with projection to the first (resp. second) component of $\Gm \times \Gm$.
Similarly we fix the isomorphism from $\Z^2$ to $X_*(S)$ which sends $(1,0)$ (resp. $(0,1)$) to the cocharacter $a \mapsto \varphi^{-1}(a,1)$ (resp. $a \mapsto \varphi^{-1}(1,a)$).

Write $\hat{S}$ for the Langlands dual torus to $S$.  It is isomorphic to $\C^\times \times \C^\times$ as a group, with $\Gamma_\R$-action defined by $\sigma(\alpha,\beta)=(\beta, \alpha)$.  We fix the isomorphism $X^*(S) \isom X_*(\hat{S})$ given by
\begin{equation*}
(a,b) \mapsto (z \mapsto (z^a,z^b)).
\end{equation*}

We have an inclusion $\iota_S: \Gm \to S$ given on $\mathcal A$-points by $a \mapsto a \otimes 1$.
Write $\sigma_S$ for the automorphism of $S$ given by $1 \otimes \sigma$ on $\mathcal A$-points.  Note that the fixed point set of $\sigma_S$ is precisely the image of $\iota_S$.

Write $\Nm: S \to \Gm$ for the norm map given by $s \mapsto s \cdot \sigma_S(s)$.  Note that the product $s \cdot \sigma_S(s)$ is in $\iota_S(\Gm)$, which we identify here with $\Gm$.  One computes that the norm map induces the map $n \mapsto (n,n)$ from $X^*(\Gm)$ to $X^*(S)$ with the above identifications.

Write $T_1$ for the kernel of this norm map. Its group of characters fits into the exact sequence
\begin{equation*}
0 \to X^*(\Gm) \to X^*(S) \to X^*(T_1) \to 0.
\end{equation*}
We identify $X^*(T_1)$ with $\Z$ in such a way so that the restriction map $X^*(S) \to X^*(T_1)$ is given by $(a,b) \mapsto a-b$. The corresponding map $\hat{S} \to \hat{T}$ is given by $(\alpha,\beta) \mapsto \alpha \beta^{-1}$.

\subsection{The Kernel and Cokernel Tori}
\begin{defn} We define $\Aker$ to be the kernel of the map from $\Gm^4 \to \Gm$ given by $(a,b,c,d) \mapsto \dfrac{ab}{cd}$.  We define $\Acok$ to be the cokernel of the map from $\Gm$ to $\Gm^4$ given by $x \mapsto (x,x,x^{-1}, x^{-1})$.  Write $\Tker$ for the kernel of the map
\[ S \times S \to \Gm \]
given by
\[ (\alpha,\beta) \mapsto \Nm(\alpha/\beta), \]
and $\Tcok$ for the cokernel of the map
\[ \Gm \to S \times S \]
given by
\[ x \mapsto (\iota_S(x),\iota_S(x^{-1})). \]
\end{defn}

Identifying $X_*(\Gm)$ and $X^*(\Gm)$ with $\Z$ as before, we obtain exact sequences
\begin{equation*}
0 \to X_*(\Aker) \to \Z^4 \to \Z \to 0,
\end{equation*}
\begin{equation*}
0 \to \Z \to \Z^4 \to X^*(\Aker) \to 0,
\end{equation*}
\begin{equation*}
0 \to \Z \to \Z^4 \to X_*(\Acok) \to 0,
\end{equation*}
\begin{equation*}
0 \to X^*(\Acok) \to \Z^4 \to \Z \to 0.
\end{equation*}

Here the maps from $\Z \to \Z^4$ are both $n \mapsto (n,n,-n,-n)$, and the maps from $\Z^4 \to \Z$ are both
$(n_1,n_2,n_3,n_4) \mapsto n_1+n_2-n_3-n_4$.

Thus we obtain isomorphisms
\begin{equation*}
g_{\kc}: X^*(A_{\ker}) \isom X_*(A_{\cok})
\end{equation*}
and
\begin{equation*}
g_{\ck}: X^*(A_{\cok}) \isom X_*(A_{\ker}),
\end{equation*}
obtained from the exact sequences defining $\Aker$ and $\Acok$.  In this way we view $\Acok(\C)$ and $\Aker(\C)$ as the dual tori $\hat{A}_{\ker}$ and $\hat{A}_{\cok}$, respectively.

The isomorphism $\varphi \times \varphi: S(\C) \times S(\C) \isom (\C^\times)^4$ gives isomorphisms $\Phi_{\ker}: \Tker(\C) \isom \Aker(\C)$ and $\Phi_{\cok}: \Tcok(\C) \isom \Acok(\C)$.

Consider the map from $S \times S$ to $S \times S$ given by $(a,b) \mapsto (ab,a \sigma_S(b))$.  This fits together with the previous maps to form an exact sequence
\begin{equation*}
1 \to \Gm \to S \times S \to S \times S \to \Gm \to 1,
\end{equation*}
and yields an isomorphism $\Psi_T: \Tcok \isom \Tker$.

Consider the map from $\Gm^4$ to $\Gm^4$ given by $(a,b,c,d) \mapsto (ac,bd,ad,bc)$.
This fit together with the previous maps to form an exact sequence
\begin{equation*}
1 \to \Gm \to \Gm^4 \to \Gm^4 \to \Gm \to 1
\end{equation*}
and yields an isomorphism $\Psi_A: \Acok \isom \Aker$.  On $\C$-points we have
\begin{equation} \label{100students}
\Phi_{\ker} \circ \Psi_T(\C)=\Psi_A(\C) \circ \Phi_{\cok}.
\end{equation}

\section{Structure of $\GSp_4(F)$} \label{Structure of}

\subsection{The General Symplectic Group}
Let $F$ be a field of characteristic $0$.  Put
\[ J=   \left( \begin{array}{cccc}
&&&1 \\
&&-1& \\
&1&& \\
-1 &&& \\  \end{array} \right). \]

Take $G$ to be the algebraic group $\GSp_4= \{ g \in \GL_4 \mid gJg^t= \mu J, \text{ some } \mu=\mu(g) \}$.  It is closely related to the group
$G'=\Sp_4= \{g \in \GSp_4 \mid \mu(g)=1 \}$.  Write $A$ for the subgroup of diagonal matrices in $G$, and $Z$ for the subgroup of scalar matrices in $G$.

We fix the isomorphism $\iota_A: \Aker \isom A$ given by
\begin{equation} \label{abcd}
(a,b,c,d) \mapsto \left( \begin{array}{cccc}
a &&& \\
&c&& \\
&& d& \\
&&& b \\  \end{array} \right).
\end{equation}

Let $B_A$ be the Borel subgroup of upper triangular matrices in $G$.

\subsection{Root Data}

Although $A$ and $\Aker$ are isomorphic tori, we prefer to parametrize their character and cocharacter groups differently, since the isomorphism $\iota_A$ permutes the order of the components.

So we express $X^*(A)=\Hom(A,\Gm)$ as the cokernel of the map
\begin{equation} \label{exactchar}
i: \Z \to \Z^4,
\end{equation}
given by $i(n)=(n,-n,-n,n)$.

Write $e_1, \ldots, e_4$ for the images in $X^*(A)$ of $(1,0,0,0), \ldots, (0,0,0,1)$.  Thus $e_1+e_4=e_2+e_3$.
The basis $\Delta_G$ of simple roots corresponding to $B_A$ is  $ \{ e_1-e_2, e_2-e_3 \}$.  The corresponding positive roots are
$\{ e_1-e_2,e_1-e_4,e_2-e_3,e_1-e_3 \}$.  The half-sum of the positive roots is then $\rho_B=\half(4e_1-e_2-3e_3) \in X^*(A)$.

\begin{defn} Write $\Omega$ for the Weyl group of $A$ in $G$.  Write $w_0,w_1$, and $w_2$ for the elements of $\Omega$ which conjugate  $\diag(a,b,c,d) \in A$ to
\begin{equation*}
\diag(d,c,b,a), \diag(a,c,b,d), \text{ and } \diag(b,a,d,c),
\end{equation*}
respectively.
\end{defn}
$\Omega$ has order $8$ and is generated by $w_0,w_1$, and $w_2$.

Express $X_*(A)$ as the kernel of the map
\begin{equation} \label{exactco}
p: \Z^4 \to \Z,
\end{equation}
given by $p(a,b,c,d)=a-b-c+d$.

Let $\vt_1=(1,0,0,-1)$ and $\vt_2=(0,1,-1,0) \in X_*(A)$.
Then the coroots of $A$ in $G$ are given by $R^\vee=R^\vee(A,G)= \{ \pm \vt_1 \pm \vt_2, \pm \vt_1, \pm \vt_2 \}$. The basis $\Delta^\vee_G$ of simple coroots dual to $\Delta_G$ is $\{  \vt_1-\vt_2, \vt_2  \}$.  Then $(X^*(A),\Delta_G, X_*(A), \Delta^\vee_G)$ is a based root datum for $G$.

\subsection{The Dual Group $\hat{G}$}

We will take $\hat{G}$ to be $\GSp_4(\C)$, with trivial $L$-action, and the same based root data as already discussed for $G$.  The isomorphism
\begin{equation} \label{Kimball}
X^*(A) \stackrel{(\iota_A)^*}{\to} X^*(\Aker) \stackrel{(\Psi_A)^*}{\to} X^*(\Acok)  \stackrel{g_{ck}}{\to}
 X_*(\Aker) \stackrel{(\iota_A)_*}{\to} X_*(A)
\end{equation}
(and its inverse) furnish the required isomorphism of based root data.  Let us write this out more explicitly.
Note that $(\iota_A)_*$ and $(\iota_A)^*$ are given by:
\begin{equation*}
(\iota_A)_*(a,b,c,d)=(a,c,d,b)
\end{equation*}
and
\begin{equation*}
(\iota_A)^*(a,b,c,d)=(a,d,b,c).
\end{equation*}

The isomorphism in (\ref{Kimball}) is induced from the linear transformation $\Sigma: \Z^4 \to \Z^4$ represented by the matrix
\[ \left( \begin{array}{cccc}
1&1&0&0 \\
1&0&1&0 \\
0&1&0&1 \\
0&0&1&1 \\  \end{array} \right), \]
which gives the exact sequence
\[ 0 \to \Z \stackrel{i}{\to} \Z^4 \stackrel{\Sigma}{\to} \Z^4 \stackrel{p}{\to} \Z \to 0, \]
and thus an isomorphism

\begin{equation} \label{bat}
 X^*(A) \stackrel{\Sigma}{\isom} X_*(A). 
\end{equation}

(This agrees with the isomorphism used in Section 2.3 in \cite{Ralf}.)

We have $\Sigma(e_1-e_2)=\vt_2$ and $\Sigma(e_2-e_3)=\vt_1-\vt_2$.
Thus the based root datum above is self-dual.  Note that $\Sigma(\rho)=\frac{3}{2} \vt_1+ \half \vt_2$.  Write $\hat A$ for $A(\C)$; it is the dual torus to $A$ via the isomorphism in (\ref{bat}).

\section{Discrete Series for $\GSp_4(\R)$} \label{Discrete Series for G}

\subsection{The maximal elliptic torus $T$ of $G$}

Consider the map $\GL_2 \times \GL_2 \to \GL_4$ given by
\begin{equation*}
 \left( \begin{array}{cc}
a &b \\
c&d  \\  \end{array} \right) \times
\left( \begin{array}{cc}
e &f \\
g&h  \\  \end{array} \right) \mapsto
\left( \begin{array}{cccc}
a &&&b \\
&e&f&  \\
&g&h& \\
c&&&d  \\  \end{array} \right).
\end{equation*}

The composition of  this with the natural inclusion $S \times S \to \GL_2 \times \GL_2$ gives an embedding of  $S \times S$ into $\GL_4$.  This restricts to an embedding of $\Tker$ into $G$, whose image is an elliptic maximal torus $T$ of $G$.  Thus we have $\iota_T: \Tker \isom T$.

$T(\R)$ is the subgroup of matrices of the form
\begin{equation} \label{gagamma}
\gamma_{r,\theta_1,\theta_2}= \left( \begin{array}{cccc}
r\cos(\theta_1) && &-r\sin(\theta_1) \\
&r\cos(\theta_2)&-r\sin(\theta_2)&  \\
&r\sin(\theta_2)&r\cos(\theta_2) & \\
r\sin(\theta_1)&&&r\cos(\theta_1)  \\  \end{array} \right),
\end{equation}
for $r>0$ and angles $\theta_1,\theta_2$.

Pick an element $\xi \in G(\C)$ so that

\begin{equation*}
\Ad(\xi) \left( \begin{array}{cccc}
a &&&-b \\
&c&-d&  \\
&d&c& \\
b&&&a  \\  \end{array} \right)=\left( \begin{array}{cccc}
a+ib &&& \\
&c+id&& \\
&& c-id& \\
&&& a-ib \\  \end{array} \right),
\end{equation*}
 and put $B_T=\Ad(\xi^{-1})B_A$.  Then $B_T$ is a Borel subgroup of $G_{\C}$ containing $T$, and $\Ad(\xi): T(\C) \isom A(\C)$ is the canonical isomorphism associated to the pairs $(T,B_T)$ and $(A,B_A)$.  The definitions have been set up so that
\begin{equation*}
\iota_A \circ \Phi_{\ker}=\Ad(\xi) \circ \iota_T.
\end{equation*}

We identify $A(\C)$ as the dual torus $\hat{T}$ to $T$ via the isomorphisms
\begin{equation} \label{Tdual}
X^*(T) \stackrel{(\iota_T)^*}{\to} X^*(\Tker)  \stackrel{\Phi_{\ker}^*}{\to} X^*(\Aker)\stackrel{(\Psi_A)^*}{\to}
 X^*(\Acok)  \stackrel{g_{ck}}{\to} X_*(\Aker)  \stackrel{(\iota_A)_*}{\to} X_*(A).
\end{equation}
 
\subsection{Real Weyl Group}

We use $\Ad(\xi)$ to identify $\Omega$ with the Weyl group of $T(\C)$ in $G(\C)$.  Recall that $\Omega_{\R}$ denotes the Weyl group of $T(\R)$ in $G(\R)$.
By \cite{Warner}, Proposition 1.4.2.1, we have 
\begin{equation*}
\Omega_\R=N_{K_{\R}}(T(\R))/(T(\R) \cap K_{\R}).
\end{equation*}

When discussing maximal compact subgroups of $\GSp_4(\R)$, it is convenient to use a different realization of these symplectic groups.  Following \cite{Bessel}, take for $J$ the symplectic matrix
\begin{equation*}
\left( \begin{array}{cccc}
 &&1& \\
&&&1  \\
-1&&& \\
&-1&&  \\  \end{array} \right).
\end{equation*}
Take for $K_{\R}$ the standard maximal compact subgroup of $\GSp_4(\R)$ (the intersection of $G(\R)$ with the orthogonal group), and $SK_{\R}$ the intersection of $K_{\R}$ with $\Sp_4(\R)$.  One finds that $SK_{\R}$ is isomorphic to the compact unitary group $U_2(\R)$, and yields the Weyl group element $w_2$.  The element $\diag(1,1,-1,-1) \in  N_{G(\R)}(T(\R)) \cap K_{\R}$ gives $w_0 \in \Omega_\R$, and these two elements generate $\Omega_\R$.  This subgroup has index $2$ in $\Omega$, and does not contain the element $w_1$.

\subsection{Admissible Embeddings} Consider the following admissible embedding $\eta_B: {}^LT \to {}^LG$.  Write $\theta(z)=\frac{z}{|z|}$ for $z \in \C^\times$. 
We have ${}^LT= \hat T  \rtimes   W_{\R}$, with $\tau$ acting as the longest Weyl group element on $\hat T$.

 Writing ${}^LT=\hat{T} \times W_\R$, we put
\begin{equation*}
\eta_B(1 \times z)=\left( \begin{array}{cccc}
\theta(z)^3 &&& \\
&\theta(z)&& \\
&&\theta(z)^{-1}& \\
&&&\theta(z)^{-3} \\  \end{array} \right) \times z
\end{equation*}
for $z \in \C^\times \cong W_\C$,

\begin{equation*}
\eta_B(\hat{t} \times 1)=\hat{t} \times 1,
\end{equation*}
for $\hat{t} \in \hat{T}$, and
\begin{equation*}
\eta_B(1 \times \tau)=J \times \tau.
\end{equation*}

\subsection{Elliptic Langlands Parameters} \label{ELP}
Let $a,b$ be odd integers with $a>b>0$.  Let $t$ be an even integer.  Put $\mu=\half [(t,t,t,t)+(a,b,-b,-a)]$ and $\nu=\half[ (t,t,t,t)+(-a,-b,b,a) ]$, viewed in $X_*(\hat{T})_\C$.  Then we may define a Langlands parameter $\varphi_G: W_\R \to {}^L G$ by
\begin{equation*}
\varphi_G(z)=z^\mu \ol{z}^{\nu} \times z=|z|^t \left( \begin{array}{cccc}
\theta(z)^a &&& \\
&\theta(z)^b&& \\
&&\theta(z)^{-b}& \\
&&&\theta(z)^{-a} \\  \end{array} \right) \times z,
\end{equation*}
and $\varphi_G(\tau)=J \times \tau$.

Note that the centralizer of $\varphi_G(W_\C)$ in $\hat{G}$ is simply $\hat A$, and that $\lip \mu, \alpha \rip$ is positive for every root of $A$ that is positive for $B_A(\C)$.  Thus $\varphi_G$ determines the pair $(\hat A, \hat B_A)$, where $\hat B_A$ is simply $B_A(\C)$.

Define a Langlands parameter $\varphi_B: W_\R \to {}^LT$ by
\begin{equation*}
\varphi_B(z)=|z|^t \diag(\theta(z)^{a-3},\theta(z)^{b-1},\theta(z)^{1-b},\theta(z)^{3-a}) \times z,
\end{equation*}
and $\varphi_B(\tau)=1 \times \tau$.  Then $\varphi_G=\eta_B \circ \varphi_B$.

Let $\pi_G= \pi(\varphi_G,B_T)$ and $\pi_G'= \pi(\varphi_G,w_1(B_T))$, recalling notation from Section \ref{Langlands packets}. The $L$-packet determined by $\varphi_G$ is
\begin{equation*}
\Pi=\{ \pi_G, \pi_G' \}.
\end{equation*}
Here $\pi_G$ is called a holomorphic discrete series representation, and $\pi_G'$ is called a large discrete series representation.

The highest weight for the associated representation $E$ of $G(\C)$ is
\begin{equation*}
\lambda_B= \half(a+b-4,t-b+1,t-a+3,0) \in X^*(A).
\end{equation*}

From this we may read off the central character $\lambda_0(zI)=z^t$ for $zI \in A_G(\C)$.

\section{The Elliptic Endoscopic Group $H$} \label{Endoscopic}

\subsection{Root Data}
  Let $H$ be the cokernel of the map $\Gm \to \GL_2 \times \GL_2$ given by $t \mapsto tI \times t^{-1}I$.
  Write $A^H$ for the diagonal matrices in $H$, and $B_H$ for the pairs of upper triangular matrices in $H$.  Fix $\iota_{A^H}: \Acok \isom A^H$ given by
\begin{equation*}
(a,b,c,d) \mapsto \left( \begin{array}{cc} a & \\
&b  \\  \end{array} \right) \times
\left( \begin{array}{cc}
d & \\
&c  \\  \end{array} \right).
 \end{equation*}

Write $T_H$ for the image of $S \times S$ in $H$.  It is an elliptic maximal torus in $H$.  Fix  $\iota_{T_H}: \Tcok \isom T_H$ obtained from the map $S \times S \to \GL_2 \times \GL_2$ given by $\alpha \mapsto (\iota_S(\alpha), \iota_S(\alpha))$.  Put $B_{T_H}=\Ad(x \times x)^{-1}B_H$, a Borel subgroup of $H_\C$ containing $T_H$.  Then $\Ad(x \times x)$ is the canonical isomorphism $T_H(\C) \isom A^H(\C)$ associated to the pairs $(T_H,B_{T_H})$ and $(A^H,B_H)$.  We view $X^*(T_H)$ as the kernel of the map $p: \Z^2 \times \Z^2 \to \Z$ given by $(a,b) \times (c,d) \mapsto a+b-c-d$.  We have a basis of roots $\Delta_H$ given by
  \begin{equation} \label{Eych}
  \Delta_H=\{ (1,-1) \times (0,0), (0,0) \times (1,-1) \},
  \end{equation}
and $\rho_H=\half(1,-1) \times \half(1,-1)$.

  Furthermore, $X_*(T_H)$ is the cokernel of the map $\iota: \Z \to \Z^2 \times \Z^2$ given by $a \mapsto (a,a) \times (-a,-a)$.  We have a basis of coroots $\Delta_H^\vee$ given by
  \begin{equation} \label{Hvee}
  \Delta_H^\vee= \{ (1,-1) \times (0,0), (0,0) \times (1,-1) \},
  \end{equation}
viewed in the quotient $X_*(T_H)$.

\subsection{Dual Group $\hat{H}$}

Let $\hat{H}= \{ (g,h) \in \GL_2(\C) \times \GL_2(\C) \mid \det(g)=\det(h) \}$.
We have an inclusion $\Aker(\C) \to \hat{H}$ given by
\begin{equation*}
(a,b,c,d) \mapsto \left( \begin{array}{cc}
a & \\
&b  \\  \end{array} \right) \times
\left( \begin{array}{cc}
d & \\
&c  \\  \end{array} \right).
\end{equation*}

Write $\hat{A}^H \subset \hat{H}$ for the image.
We thus have an isomorphism $\iota_{\hat{A}^H}: \Aker(\C) \isom \hat{A}^H$.

Also write $\hat{B}_H$ for the subgroup of upper triangular matrices in $\hat{H}$.  This Borel subgroup determines a based root datum for $\hat{H}$.

Giving $\hat{H}$ the trivial $L$-action, we view it as a dual group to $H$ via the isomorphisms
\begin{equation*}
X^*(A^H) \stackrel{(\iota_{A^H})^*}{\to} X^*(\Acok) \stackrel{g_{\ck}}{\to} X_*(\Aker) \stackrel{(\iota_{\hat{A}^H})_*}{\to} X_*(\hat{A}^H), 
\end{equation*}

\begin{equation*}
X^*(\hat{A}^H) \stackrel{(\iota_{\hat{A}^H})^*}{\to} X^*(\Aker) \stackrel{g_{\kc}}{\to} X_*(\Acok) \stackrel{(\iota_{A^H})_*}{\to} X_*(A^H).
\end{equation*}

We identify $\hat{A}^H$ as the dual torus $\hat{T}_H$ to $T_H$ via the isomorphisms
\begin{equation} \label{THdual}
X^*(T_H)\stackrel{(\iota_{T_H})^*}{\to} X^*(\Tcok) \stackrel{\Phi_{\cok}^*}{\to} X^*(\Acok) \stackrel{g_{ck}}{\to}
X_*(\Aker) \stackrel{(\iota_{\hat{A}^H})^*}{\to} X_*(\hat{A}^H).
\end{equation}

Let $\eta: {}^LH \to {}^LG$ be given by
\begin{equation} \label{eta}
 \left( \begin{array}{cc}
a &b \\
c&d  \\  \end{array} \right) \times
\left( \begin{array}{cc}
e &f \\
g&h  \\  \end{array} \right) \times w \mapsto
\left( \begin{array}{cccc}
a &&&b \\
&e&f&  \\
&g&h& \\
c&&&d  \\  \end{array} \right) \times w.
\end{equation}

Let
\begin{equation*}
s=\left( \begin{array}{cc}
1 & \\
&1  \\  \end{array} \right) \times
\left( \begin{array}{cc}
-1 & \\
&-1  \\  \end{array} \right)
 \in \hat{H}.
 \end{equation*}

The image $\eta(\hat{H})$ is the connected centralizer in $\hat{G}$ of $\eta(s)$.
Thus, $(H,s,\eta)$ is an elliptic endoscopic triple for $G$.  In fact it is the only one, up to isomorphism.

Moreover note that $\eta$ restricted to $\hat{A}^H$ is given by
\begin{equation} \label{hulk}
\eta|_{\hat{A}^H}=\iota_A \circ (\iota_{\hat{A}^H})^{-1}.
\end{equation}
(Recall that $\hat{A}=A(\C)$.)

\section{Transfer for $H(\R)$}  \label{Discrete Series for H}

The goal of this section is Proposition \ref{TransferToH}, in which we identify $e_{\pi_G}^H$ and $e_{\pi_G'}^H$.  This is part of the global transfer $f^Hdh$ which is to be entered into $ST_g$ for the endoscopic group $H$.  We will recognize it using the character theory of transfer reviewed in Section \ref{Transfer}.

\subsection{Parametrization of Discrete Series}

First we must set up the Langlands parameters for discrete series representations of $H(\R)$, and describe how they transfer to $L$-packets in $G(\R)$.  Recall that we have fixed three integers $a,b,t$, with $a,b$ odd, $t$ even, and $a>b>0$.  Define the Langlands parameter $\varphi_H: W_\R \to {}^LH= \hat H \times W_{\R}$ by
\begin{equation*}
\varphi_H(z)=|z|^t\left( \begin{array}{cc}
\theta(z)^a & \\
&\theta(z)^{-a}  \\  \end{array} \right) \times
|z|^t \left( \begin{array}{cc}
\theta(z)^b & \\
&\theta(z)^{-b}  \\  \end{array} \right) \times z
\end{equation*}
for $z \in W_{\C}$, and
\begin{equation*}
\varphi_H(\tau)=\left( \begin{array}{cc}
  & 1\\
-1& \\  \end{array} \right) \times
 \left( \begin{array}{cc}
 & -1 \\
1&  \\  \end{array} \right) \times \tau.
\end{equation*}

Then $\varphi_H$ determines the pair $(\hat{A}_H,\hat{B}_H)$.  The $L$-packet is a singleton $\{ \pi_H \}$. The corresponding representation $E_H$ of $H(\C)$ has highest weight
\begin{equation*}
\lambda_H= \half(t+a-1,t-a+1) \times \half(t+b-1,t-b+1).
\end{equation*}
From this we read off the central character $\lambda_0^H(z_1,z_2)=(z_1z_2)^t$.  Most importantly, we have $\varphi_G=\eta \circ \varphi_H$.

There is another Langlands parameter $\varphi_H'$ given by
\begin{equation*}
\varphi_H'(z)=|z|^t\left( \begin{array}{cc}
\theta(z)^b & \\
&\theta(z)^{-b}  \\  \end{array} \right) \times
|z|^t \left( \begin{array}{cc}
\theta(z)^a & \\
&\theta(z)^{-a}  \\  \end{array} \right) \times z,
\end{equation*}
and by $\varphi_H'(\tau)=\varphi_H(\tau)$ as above.

Again the $L$-packet is a singleton $\{ \pi_H' \}$.  The corresponding representation $E_H'$ has highest weight
\begin{equation*}
\lambda_H'=\half(t+b-1,t-b+1) \times \half(t+a-1,t-a+1),
\end{equation*}
and the same central character $\lambda_0^H$ as $E_H$.

Let $\varphi_G'=\eta \circ \varphi_H'$.  Then $\varphi_G'=\Int(w_2) \circ \varphi_G$, so it is equivalent to $\varphi_G$.   In particular, both $L$-packets $\{ \pi_H \}$ and $\{ \pi_H' \}$ transfer to $\Pi=\{ \pi_G, \pi_G' \}$.

\subsection{Alignment}

Recall the definition of alignment from Section \ref{Transfer}.

\begin{lemma} \label{Pallavi} Define $j: T_H \isom T$ by $j=\iota_T \circ \Psi_T \circ (\iota_{T_H})^{-1}$.   Then $(j,B_T,B_{T_H})$ is aligned with $\varphi_H$, and $(j,w_1B_T,B_{T_H})$ is aligned with $\varphi_H'$. \end{lemma}

\begin{proof}
Since the parameter $\varphi_G$ gives the pair $(\hat{A},\hat{B})$, the parameter $\varphi_G'$ gives the pair $(\hat{A},w_1 \hat{B})$, and since $\varphi_H$ and $\varphi_H'$ both give $(\hat{A},\hat{B}_H)$, the horizontal maps in (\ref{alignment})
are identities.
The map $\hat{j}:  \hat{T} \to \hat{T}_H$ may be computed by composing the isomorphism $X_*(\hat{T}) \isom X^*(T)$ in (\ref{Tdual}) with the induced map $j^*: X^*(T) \isom X^*(T_H)$ and finally with the inverse of the isomorphism $X_*(\hat{T}_H) \isom X^*(T_H)$ in (\ref{THdual}).  Using Equations (\ref{100students}) and (\ref{hulk}), one finds that
\begin{equation*}
\hat{j}=\iota_{\hat{A}_H} \circ (\iota_A)^{-1}= \eta^{-1},
\end{equation*}
as desired.

\end{proof}

\subsection{Transfer for $H_\R$}

\begin{prop} \label{TransferToH}  Let $\pi_G=\pi(\varphi_G,B_T)$ and $\pi_G'= \pi(\varphi_G, \omega^{-1}(B_T))$ as described in Section \ref{ELP}.  Then (using notation from Section \ref{PC}) we may take $e_{\pi_G}^H=e_{\pi_H}+ e_{\pi_H'} $, where $\pi_H$ (resp., $\pi_H'$) is the discrete series representation determined by $\varphi_H$ (resp., by $\varphi_H'$) as above.  Furthermore, we may take $e_{\pi_G'}^H=-e_{\pi_G}^H$.
\end{prop}

\begin{proof}
By Lemma \ref{Pallavi}, we may use 
\begin{equation*}
\Delta_{\infty}(\varphi_H,\pi(\varphi_G,\omega^{-1}(B_T)))= \lip a_\omega, \hat j^{-1}(s) \rip
\end{equation*}
and 
\begin{equation*}
\Delta_{\infty}(\varphi_H',\pi(\varphi_G,\omega^{-1}(w_1B_T)))= \lip a_{w_1 \omega}, \hat j^{-1}(s) \rip
\end{equation*}
for $\omega \in \Omega$.
In both cases, this is given by   
\begin{equation*}
\lip a_{\omega},s \rip= \begin{cases} & 1, \text{ if } \omega \in \Omega_{\R}, \\
              & -1, \text{ if } \omega \notin \Omega_{\R}. \end{cases}
\end{equation*}
Note that $ \lip a_{w_1 \omega}, \hat j^{-1}(s) \rip=- \lip a_\omega, \hat j^{-1}(s) \rip$.
Therefore the characterization (\ref{charidentity}) becomes, for a general measure $f_{\infty}dg_{\infty}$ at the real place,

\begin{equation*}
\begin{split}
\Theta_{\pi_H}(f_{\infty}^Hdh_{\infty}) &= \sum_{\pi \in \Pi(\varphi_G)} \Delta_{\infty}(\varphi_H,\pi) \Theta_{\pi}(f_{\infty}dg_{\infty}) \\
&=\Theta_{\pi_G}(f_{\infty} dg_{\infty})-\Theta_{\pi_G'}(f_{\infty}dg_{\infty})
\end{split}
\end{equation*}
and similarly
\begin{equation*}
 \Theta_{\pi_H'}(f_{\infty}^Hdh_{\infty})  =\Theta_{\pi_G}(f_{\infty}dg_{\infty})-\Theta_{\pi_G'}(f_{\infty}dg_{\infty}).
\end{equation*}
In our case, we obtain 
\begin{equation*}
\Theta_{\pi_H}(e_{\pi_G}^H)=\Theta_{\pi_H'}(e_{\pi_G}^H)=(-1)^{q(G)},
\end{equation*}
and
\begin{equation*}
\Theta_{\pi_H}(e_{\pi_G'}^H)=\Theta_{\pi_H'}(e_{\pi_G'}^H)=-(-1)^{q(G)}.
\end{equation*}

The proposition follows.
\end{proof}

\section{Levi Subgroups} \label{Levi Subgroups}

\subsection{Levi Subgroups}

We give the standard Levi subgroups of $G$, which are those of the parabolic subgroups containing $B_A$.  We have the group $A$, the group $G$ itself, and the following two Levi subgroups:

\begin{equation*}
M_1= \left\{ \left( \begin{array}{cc}
g & \\
& \lambda g \\  \end{array} \right) \mid g \in \GL_2, \lambda \in \Gm \right\}.
\end{equation*}

\begin{equation*}
M_2= \left\{ \left( \begin{array}{ccc}
a && \\
&g& \\
&& b \\  \end{array} \right) \mid g \in \GL_2, a,b \in \Gm, \det(g)=ab \right\}.
\end{equation*}
 Note that both $M_1$ and $M_2$ are isomorphic to $\Gm \times \GL_2$.

The group $H$ also has four Levi subgroups, namely $A^H$, the group $H$ itself, the image $M_1^H$ of $\GL_2 \times A_0$ in $H$, and the image $M_2^H$ of $A_0 \times \GL_2$ in $H$.  Note that both $M_1^H$ and $M_2^H$ are isomorphic to $\GL_2 \times \Gm$.

\subsection{Miscellaneous constants}

We now compute the invariants from Section \ref{Various Invariants} for the Levi subgroups of $G$ and $H$.

First, we compute the various $k(M)$.  When $M$ is the split torus $A$ its derived group is trivial and so $k(A)=1$.  For $i=1,2$, the Levi subgroup $M_i$ is isomorphic to $\GL_2 \times \Gm$, and the torus is isomorphic to $S \times \Gm$.
Since $S$ and $\Gm$ have trivial first cohomology, again $k(M_1)=1$.

\begin{lemma} We have $k(G)=2$.
\end{lemma}

Write $T$ as before for the elliptic torus of $G$.

\begin{proof}
Recall that $T_1$ is the kernel of $\Nm$ and $H^1(\R,T_1)$ has order $2$.

Recall that the torus $T$ is isomorphic to the kernel of the map
\[ S \times S \to \Gm \]
given by
\[ (\alpha,\beta) \mapsto \Nm(\alpha/\beta). \]
Projection to the first (or second) component followed by $\Nm$ gives an exact sequence
\begin{equation} \label{T_0}
 1 \to T_1 \times T_1 \to T \to \Gm \to 1.
\end{equation}

We have that $G_{\simp}=G_{\der}$ and the inclusion $T_{\simp} = G_{\der} \cap T \subset T$ may be identified with the map $T_1 \times T_1 \to T$ in the sequence above. In particular, $H^1(\R,T_{\simp})$ has order $4$.

Taking the cohomology of (\ref{T_0}) gives the exact sequence
\[ 1 \to \R^\times/\R^{\times 2} \to H^1(\R,T_{\simp}) \to H^1(\R,T) \to 1, \]
from which we conclude that $H^1(\R,T_{\simp}) \to H^1(\R,T)$ is surjective and $H^1(\R,T)$ has order $2$.
\end{proof}

One must also compute $k(M_H)$ for Levi subgroups $M_H$ of $H$.  The intermediate Levi subgroups are again isomorphic to $\GL(2) \times \Gm$, and for $A_H$ the derived group is trivial.  So $k(M_H)=1$ for each of these.

\begin{lemma} We have $k(H)=1$. \end{lemma}
\begin{proof}
We have $T=P(S \times S)$, $H_{\simp}=\SL_2 \times \SL_2$, and $T_{\simp}=T_1 \times T_1$.  The map $T_{\simp} \to T$ factors through $T_1 \times T_1 \to S \times S$.  As above we conclude that $k(H)=1$.
\end{proof}

Secondly, we compute the Tamagawa numbers.  Recall that
\begin{equation*}
\tau(G)=|\pi_0(Z(\hat{G})^{\Gamma_\Q})| \cdot |\ker^1(\Q,Z(\hat{G}))|^{-1}.
\end{equation*}

\begin{prop} We have $\tau(M)=1$ for all Levi subgroups of $G$, and for all proper Levi subgroups of $H$, and $\tau(H)=2$. \end{prop}

\begin{proof} For each of these groups, $Z(\hat{M})$ is either the group $\C^\times$ with trivial $\Gamma_\Q$-action, or a product of such groups.  By the Chebotarev Density Theorem, the homomorphism

\begin{equation*}
\Hom(\Gamma_{\Q},\C^\times) \to \prod_v \Hom(\Gamma_{\Q_v},\C^{\times})
\end{equation*}

is injective.  So $|\ker^1(\Q,Z(\hat{G}))|$ is trivial for our examples.  Computing the component group of each $Z(\hat{M})$ is straightforward.

\end{proof}

The quantities $n^G_M$ are easy to compute, using $N_G(M) \subseteq N_G(Z(M))$.  When $M$ is a maximal torus, $n^G_M$ is of course the order of the Weyl group.
For the intermediate cases, one finds that $n^G_{M_i}=n^H_{M^H_i}=2$.
 
For $\gm=1$, we have $\ol{\iota}^M(\gm)=1$ for each $M$, since each $M$ is connected.
Note that for Levi subgroups $M$ of $G$, all proper Levi subgroups $M$ of $H$, and all semisimple elements $\gm$ in $G$ or $H$,
we have $\ol{\iota}^M(\gm)=1$ since in all these cases the derived groups are simply connected.

Finally, we compute $\iota(G,H)$, which we recall is given by
\begin{equation*}
\iota(G,H)= \tau(G) \tau(H)^{-1} |\Out(H,s,\eta)|^{-1}.
\end{equation*}
One may compute the order of $\Out(H,s,\eta)$ through Section 7.6 of \cite{K84}, which shows that this set is in bijection with $\bigwedge(\eta(s),\rho)$, in the notation of that paper.  This last set is represented by $\{1,g \}$, where
\begin{equation*}
g= \left( \begin{array}{cccc}
&1&& \\
1&&&  \\
&&&1 \\
&&1&  \\  \end{array} \right).
\end{equation*}

The conclusion is that $\iota(G,H)=\dfrac{1}{4}$.

\section{Computing $S\Phi_M$ for Levi Subgroups of $G$} \label{Computing for G Levis}

Recall from Proposition \ref{MPI} the formula

\begin{equation*} 
 \Phi_M(\gm, \Theta^E)=(-1)^{q(L)}|\Omega_L| \sum_{\omega \in \Omega^{LM}} \eps(\omega) \tr(\gm; V^M_{\omega(\lambda_B+\rho_B)-\rho_B}),
\end{equation*}
for $\gm \in T_e(\R)$.

In this section, the maximal torus will be conjugate to $A$, and the character group will be identified with $X^*(A)$.   We specify an inner product we use on $X^*(A)_\R$ for the Weyl Dimension Formula  (Proposition \ref{Weyl}).  

\begin{defn} The usual dot product gives an inner product $(,)$ on $X_*(A)_\R$, viewing it as a hypersurface in $\R^4$.

Consider the isomorphism
\begin{equation*}
\pr: X^*(A)_\R \isom X_*(A)_\R
\end{equation*}
given by
\begin{equation*}
\pr(a,b,c,d)=(a,b,c,d)-\frac{a+d-b-c}{4}(1,-1,-1,1),
\end{equation*}
and let
\begin{equation*}
\lip \lambda, \mu \rip=(\pr(\lambda),\pr(\mu)).
\end{equation*}
\end{defn}

For instance,
\begin{equation*}
\pr(\lambda_B)=\frac{1}{4} (a+b+t-4,a-b+t-2,-a+b+t+2,-a-b+t+4).
\end{equation*}
It will also be necessary to compute $\Omega^{LM}$ for each example. Recall that this is the set of $w \in \Omega$ so that $w^{-1}\alpha >0$ for positive roots $\alpha$ which are either real or imaginary.

\subsection{The term $\Phi_G$}

By Proposition \ref{SpalPhiformula} we have $\Phi_G(\gm,\Theta^E)=\tr(\gm; E)$.  Using the Weyl dimension formula, we compute
\[ S\Phi_G(1,e_{\pi_G})= -\frac{1}{24} ab(a+b)(a-b) \ol{v}(G)^{-1}. \]

\subsection{The term $S\Phi_{M_1}$}
Consider the torus $T_{M_1}$ given by
\[ \left(\begin{array}{cccc}
a & b & & \\
-b & a && \\
&& \lambda a & \lambda b\\
&& -\lambda b & \lambda a\\
\end{array} \right), \]
with $a^2+b^2 \neq 0$ and $\lambda \neq 0$.  This is an elliptic torus in $M_1$.

There is one positive real root, $e_1-e_3$, and one positive imaginary root, $\alpha_{M_1}=e_1-e_2$.  We have $\Omega^{LM}=\{1,w_1\}$, $q(L)=1$, and $|\Omega_L|=2$.  This gives
\begin{equation*}
 \Phi_{M_1}(1, \Theta^E)= (-2) \left[ \dim_\C V^{M_1}_{\lambda_B}- \dim_\C V^{M_1}_{\lambda_B'} \right],
\end{equation*}
where $\lambda_B'=\half(a+b-4,t-a+1,t-b+3,0)\in X^*(T)$.

Note that $\lip \alpha_{M_1},\lambda_B \rip=\half (b-1)$.  The Weyl dimension formula yields $\dim_{\C} V_{\lambda_B}^{M_1}=b$ and $\dim_{\C} V_{\lambda_B'}^{M_1}=a$.
Thus,
\[ S \Phi_{M_1}(1,e_{\pi_G})=-(b-a)\ol{v}(M_1)^{-1}. \]

\subsection{The term $S\Phi_{M_2}$}

Consider the torus $T_{M_2}$ given by
\[ \left(\begin{array}{cccc}
s &  & & \\
 & a &-b& \\
&b&a  & \\
&&  & t\\
\end{array} \right), \]
with $st=a^2+b^2 \neq 0$. This is an elliptic torus in $M_2$.

We may conjugate this in $G(\C)$ to matrices of the form
\begin{equation*}
\gamma= \left( \begin{array}{cccc}
s &&& \\
&a+ib&&  \\
&& a-ib& \\
&&& t  \\  \end{array} \right)
\end{equation*}
in $A(\C)$.  Composing the roots of $A$ with this composition, we determine the positive imaginary root $\alpha_{M_2}=e_2-e_3$.
We have $\Omega^{LM}=\{1,w_2 \}$.

This gives
\begin{equation*}
 \Phi_{M_2}(1, \Theta^E)= (-2) \left[ \dim_\C V^{M_2}_{\lambda_B}- \dim_\C V^{M_2}_{\lambda_B''} \right],
\end{equation*}
where $\lambda_B''=\half(t-b-1,a+b-2,0,t-a+3)\in X^*(T)$.  Note that
\begin{equation*}
\pr(\lambda_B'')=\frac{1}{4}(t+a-b-4,t+a+b-2,t-a-b+2,t-a+b+4).
\end{equation*}

 The Weyl dimension formula yields $\dim_{\C} V_{\lambda_B}^{M_2}=\half(a-b)$ and $\dim_{\C} V_{\lambda_B''}^{M_2}=\half(a+b)$, and so
\begin{equation*}
S\Phi_{M_2}(1, e_{\pi_G})=b \cdot \ol{v}(M_2)^{-1}.
\end{equation*}

\subsection{The term $S\Phi_A$}

By Proposition \ref{SpalPhiformula}, we have $\Phi_A(1, \Theta^E)= (-1)^{q(G)} |\Omega_G|=-8$, and so
\begin{equation*}
S \Phi_A(1, e_{\pi_G})=4 \ol{v}(A)^{-1}.
\end{equation*}

\section{Computing $S\Phi_{M_H}$ for Levi subgroups of $H$} \label{Computing for H Levis}

Since $e_{\pi_G}^H=e_{\pi_H}+e_{\pi_H'}$, we have
\begin{equation*}
S \Phi_{M_H}(1,e_{\pi_G}^H )=(-1)^{q(G)}(-1)^{\dim(A_{M_H}/A_H)}\ol{v}(M_H)^{-1} \left[ \Phi_{M_H}(1,\Theta_{\pi_H})+ \Phi_{M_H}(1,\Theta_{\pi_H'}) \right].
\end{equation*}

\subsection{The term $S\Phi_H(1,e_{\pi_G}^H)$}

In this case $H$ has the elliptic torus $T_H$.

From Proposition \ref{SpalPhiformula}, we obtain
$\Phi_H(1,\Theta_{\pi_H})=\dim_\C E_H$.  To apply the dimension formula, we compute for instance $\lip \alpha_1,\lambda_H \rip=a-1$, $\lip \alpha_2, \lambda_H \rip = b-1$, and $\lip \alpha_i,\rho_H \rip =1$.

We find that
\begin{equation*}
\Phi_H(1,\Theta^{E_H})=\Phi_H(1,\Theta^{E_H'})=ab.
\end{equation*}

Therefore
\begin{equation*}
S \Phi_H(1,e_{\pi_G}^H)=-2 \ol{v}(H)^{-1}ab.
\end{equation*}

\subsection{The term $S\Phi_{A^H}(1,e_{\pi_G}^H)$}

From Proposition \ref{SpalPhiformula}, we obtain
\begin{equation*}
\Phi_{A^H}(1,\Theta^{E_H})=\Phi_{A^H}(1,\Theta^{E_H'})=4.
\end{equation*}
Therefore
\begin{equation*}
S \Phi_{A^H}(1,e_{\pi_G}^H)=-8 \ol{v}(A^H)^{-1}.
\end{equation*}

\subsection{The terms $S\Phi_{M_H}(1,e_{\pi_G}^H)$ for the intermediate Levi subgroups}

 For both $M=M_H^1$ and $M=M_{H}^2$, we have $\Omega_G=\Omega_L \Omega_M$, and so formula (\ref{SpalPhiformula}) becomes simply $\Phi_{M_H}(1,\Theta^{E_H})=(-2)\dim_\C V_{\lambda_H}^{M_H}$ for both of these Levi subgroups.

We obtain
\begin{equation*}
\Phi_{M_H^1}(1,\Theta^{E_H})=  \Phi_{M_H^2}(1,\Theta^{E_H'})= -2a
\end{equation*}
and
\begin{equation*}
\Phi_{M_H^2}(1,\Theta^{E_H})=  \Phi_{M_H^1}(1,\Theta^{E_H'})= -2b.
\end{equation*}

Therefore
\begin{equation*}
S \Phi_{M_H^1}(1,e_{\pi_G}^H)=S \Phi_{M_H^2}(1,e_{\pi_G}^H)=-2\ol{v}(M_H^1)^{-1}(a+b).
\end{equation*}

\section{Final Form:  $\gm$ central} \label{Final Central}

Recall that $G=\GSp_4$.  For the convenience of the reader, we recall the set-up.

Let $a,b$ be odd integers with $a>b>0$, and $t$ an even integer.    Consider the Langlands parameter $\varphi_G: W_\R \to {}^L G$ given by
\begin{equation*}
\varphi_G(z)=|z|^t \left( \begin{array}{cccc}
\theta(z)^a &&& \\
&\theta(z)^b&& \\
&&\theta(z)^{-b}& \\
&&&\theta(z)^{-a} \\  \end{array} \right) \times z,
\end{equation*}
and $\varphi_G(\tau)=J \times \tau$.

Let $\pi_G$ be the discrete series representation $\pi(\varphi_G,B_T)$ of $G(\R)$ as in Section \ref{Langlands packets}.  Write $\pi_G'$ for the other representation in $\Pi(\varphi_G)$.

Put $f_{\infty}dg_{\infty}=e_{\pi_G}$ as in Section \ref{PC} for $\pi_G$ and any measure $f^\infty dg_f$ on $G(\Aff_f)$.  Let $fdg=e_{\pi_G} f^\infty  dg_f$, a measure on $G(\Aff)$.  By the theory of endoscopic transfer there is a matching measure $f^H dh$ on $H(\Aff)$, where $H$ is the elliptic endoscopic group $P(\GL_2 \times \GL_2)$ discussed above.

If $z \in A_G(\Q)$, then $\sum_M ST_g(fdg,z , M)$ is given by the product of $\lambda_0(z)=z^t$ with:
\begin{equation*}
-\frac{1}{24} ab(a+b)(a-b) \ol{v}(G)^{-1} f^\infty(z)+ \half (a-b)\ol{v}(M_1)^{-1}f^\infty_{M_1}(z)+ \half b \ol{v}(M_2)^{-1}f^\infty_{M_2}(z)+\half \ol{v}(A)^{-1} f^\infty_A(z).
\end{equation*}
If $z=(z_1,z_2) \in A_H(\Q)$, then $\sum_{M_H} ST_g(f^Hdh,z , M_H)$ is given by the product of $\lambda_0^H(z)=(z_1z_2)^t$ with:
\begin{equation*}
-4ab \ol{v}(H)^{-1}f^{H,\infty}(z)-2(a+b)\ol{v}(M_H^1)^{-1}f^\infty_{M_2}(z)-  2\ol{v}(A^H)^{-1}f^\infty_{A^H}(z).
\end{equation*}

\section{The case $\Gamma=\Sp_4(\Z)$} \label{Integral Points}

Let $f^\infty dg_f=e_{K_0}$, where $K_0=G(\OO_f)$.  Here $dg_f$ is an arbitrary Haar measure on $G(\Aff_f)$ so that $dg=dg_f dg_{\infty}$ is the Tamagawa measure on $G(\Aff)$.

\subsection{Central terms in $G$}
Note that $f^\infty_{M}(z)=0$ for all $z \in Z(\Q)$ unless $z= \pm 1$, and that $f_M^\infty(1)=f_M^\infty(-1)$ for all Levi subgroups $M$.

First we compute $ST_g(f dg,\pm 1, G)$.  We have
\begin{equation*}
\begin{split}
-\frac{1}{2^33} ab(a+b)(a-b) \ol{v}(G)^{-1} f^\infty(\pm 1) &= -\frac{1}{2^33} ab(a+b)(a-b) \tau(G)^{-1} d(G)^{-1} \chi_{K_0}(G) \\
 &= 2^{-10}3^{-3}5^{-1}ab(a+b)(a-b).
\end{split}
\end{equation*}

Next we treat the $\pm 1 \in M_i$ terms, for the intermediate Levi subgroups.
We have
\begin{equation*}
\begin{split}
ST_g(f dg,\pm 1, M_1) &=\half (a-b)\ol{v}(M_1)^{-1}f^\infty_{M_1}(\pm 1) \\
           &= -2^{-5}3^{-1}(a-b),
\end{split}
\end{equation*}
and
\begin{equation*}
\begin{split}
ST_g(f dg,\pm 1 \in M_2] &= \half b \ol{v}(M_2)^{-1}f^\infty_{M_2}(\pm 1) \\
&=-2^{-5}3^{-1}b.
\end{split}
\end{equation*}

Next we treat the $\pm 1 \in A$ terms.  We have $f_A(1)=\vol_{da_f}(K \cap A(\Aff_f))^{-1}$, which is $1$.  Moreover we take Lebesgue measure on $A(\R)$ so that $\ol{v}(A)=8$.  It follows that
\begin{equation*}
ST_g(f dg,\pm 1 , A)= \half \ol{v}(A)^{-1} f^\infty_A(\pm 1)=2^{-4}.
\end{equation*}

Doubling these terms to account for both central elements, we compute 

\begin{equation} \label{stablepart}
\sum_{z, M} ST_g(f dg,z , M)=2^{-9}3^{-3}5^{-1}ab(a+b)(a-b)-2^{-4}3^{-1}(a-b)-2^{-4}3^{-1}b+2^{-3}.
\end{equation}

\subsection{Central Terms in H}
By the Fundamental Lemma (See Hales \cite{Hales} and Weissauer \cite{Weissauer} for $\GSp_4$,  of course Ng\^{o} \cite{Ngo} in general), we may write $(e_{K_0})^H=e_{K_H}$, where $K_H=H(\OO_f)$.  Thus $(f^\infty)_M^H(z)=0$ for all $z \in H(\Q)$ unless $z=(1,\pm 1)$, and
\begin{equation*}
f_M^{H \infty}(1,1)=f_M^{H \infty}(1,-1)
\end{equation*}
for all Levi subgroups $M=M_H$ of $H$.

The only nontrivial factors in the formula of Theorem \ref{chi-theorem} are $|\ker \rho(\Q)|=2$, $[H(\R): H(\R)_+]=4$, and $\chi_{\alg}(H^{\simp}(\Z))$.  Note that $H^{\simp}=\SL_2 \times \SL_2$.

Therefore
\begin{equation*}
\begin{split}
\chi_{K_H}(H) &= 2^{-1} \chi_{\alg}(\SL_2(\Z))^2\\
                  &= 2^{-5}3^{-2}.
 \end{split}
 \end{equation*}

We conclude that

\begin{equation*}
\begin{split}
ST_g(f^Hdh,(1,\pm 1), H) &=-4ab \ol{v}(H)^{-1}\vol(K_H)^{-1} \\
&= -2^{-4} 3^{-2}ab.
\end{split}
\end{equation*}

Next we find that $\sum_{i=1}^2 ST_g(f^H dh,(1,\pm 1), M_i^H)$ is equal to
\begin{equation*}
\begin{split}
\sum_{i=1}^2 ST_g(f^H dh,(1,\pm 1), M_i^H) &=-2(a+b)\ol{v}(M_1^H)^{-1}\vol(K_M)^{-1} \\
&=2^{-3} 3^{-1}(a+b). 
\end{split}
\end{equation*}

Finally, we have

\begin{equation*}
\begin{split}
 ST_g(f^H dh,(1,\pm 1) , A^H) &= - 2 \ol{v}(A)^{-1}\vol (K_A)^{-1} \\
 &=-2^{-2}.
\end{split}
\end{equation*}

Multiplying by $\iota(G,H)=4^{-1}$, then doubling to account for both central elements, we compute 

\begin{equation} \label{labilepart}
\iota(G,H)\sum_{z,M_H} ST_g(f^H dh, z , M_H)= -2^{-5} 3^{-2}ab +2^{-4} 3^{-1}(a+b) -2^{-3}.
\end{equation}

\section{Comparison} \label{Comparison}

As mentioned in the introduction, Wakatsuki in \cite{Wak Mult}, \cite{Wak Dim} has used the Selberg Trace Formula, and Arthur's $L^2$-Lefschetz number formula to compute the discrete series multiplicities $m_{\disc}(\pi, \Gamma)$ for $\pi$ both holomorphic and large discrete series representations for $\Sp_4(\R)$, and for many cases of arithmetic subgroups $\Gamma$.   We will compare our formula to his when $\Gamma$ is the full modular group.  (Note that if $\pi$ is a discrete series representation of $\GSp_4(\R)$ with trivial central character, and $\pi_1$ is its restriction to $\Sp_4(\R)$, then $m_{\disc}(\pi,\Gamma)=m_{\disc}(\pi_1, \Gamma_1)$, where $\Gamma_1= \Sp_4(\Z)$.)  Since he is using the Selberg trace formula, his formula breaks into contributions from each conjugacy class in $\Gamma$.  In particular, he identifies the central-unipotent contributions $H_1^{\Hol}$ and $H_1^{\largess}$ to $m_{\disc}(\pi_G)$ and $m_{\disc}(\pi_G')$, respectively.  Namely, 

\begin{equation*}
H_{1}^{\Hol}=2^{-9}3^{-3}5^{-1}ab(a-b)(a+b)-2^{-5}3^{-2}ab+2^{-4}3^{-1}b
\end{equation*}
and
\begin{equation*}
H_{1}^{\largess}=2^{-9}3^{-3}5^{-1}ab(a-b)(a+b)+2^{-5}3^{-2}ab-2^{-3}3^{-1}b+2^{-2}.
\end{equation*}

(To translate from his notation to ours, use $j=b-1$ and $k= \half(a-b)+2$.)

Using these formulas and and our formulas above, we find

\begin{equation*}
H_{1}^{\Hol}= \sum_M ST_g(f dg, \pm 1, M)+\iota(G,H)\sum_{M_H} ST_g(f^H dh, (1, \pm 1),M_H)
\end{equation*}
when $fdg=e_{\pi_G} e_{K_0}$
and
\begin{equation*}
H_{1}^{\largess}= \sum_M ST_g(f dg, \pm 1, M)+\iota(G,H)\sum_{M_H} ST_g(f^H dh, (1, \pm 1),M_H).
\end{equation*}
when $fdg=e_{\pi_G'} e_{K_0}$.

This proves Theorem \ref{Theorem}.

\bibliographystyle{plain}

\end{document}